\documentclass[a4paper,11pt]{article}
\usepackage[utf8]{inputenc}
\usepackage{latexsym}
\usepackage{amsfonts}
\usepackage{amssymb}

\usepackage{enumitem} 
\usepackage{braket}
\usepackage{hyperref}

\usepackage{amsmath}
\usepackage{amssymb}
\usepackage{amsfonts}
\usepackage{amsthm}
\usepackage{indentfirst}
\usepackage{psfrag}
\usepackage{amscd,stmaryrd,latexsym}
\usepackage[margin=1.3in]{geometry}
\usepackage[pdftex]{graphicx}
\newtheorem{thm}{Theorem}[section]
\newtheorem{lem}{Lemma}[section]

\newtheorem{cor}{Corollary}[section]

\newtheorem{Prop}{Proposition}[section]
\newtheorem*{St*}{Statement}

\newtheorem*{theorem*}{Theorem}
\theoremstyle{remark}

\theoremstyle{definition}

\theoremstyle{remark}
\newtheorem{oss}{Remark}[section]

\newcommand{\be}{\begin{equation}}
\newcommand{\ee}{\end{equation}}
\newcommand{\R}{\mathbb{R}}

\newcommand{\N}{\mathbb{N}}

\newcommand{\Het}{\mathbb{H}}
\newcommand{\OHet}{\overline{\mathbb{H}}}

\newcommand{\He}{\mathbb{H}_{\varepsilon}}

\newcommand{\ca}[2]{\mathcal{#1}_{#2}}

\newcommand{\spt}[1]{\text{spt}\,\|#1\|}

\newcommand\res{\mathop{\hbox{\vrule height 7pt width .5pt depth 0pt
\vrule height .5pt width 6pt depth 0pt}}\nolimits}

\def\eps{\mathop{\varepsilon}}

\def\Ece{\mathop{\mathcal{E}_{\varepsilon}}}

\def\Hc{\mathop{\mathcal{H}}}

\def\om{\omega}

\def\p{\partial}

\def\quo{\tilde{M}\stackrel{\times}{\sim}[0,\om)}

\def\Se{\mathcal{S}_{\varepsilon}}

\def\eps{\mathop{\varepsilon}}

\def\om{\omega}
\def\p{\partial}

\newcommand{\Rc}[1]{\text{Ric}_{#1}} 
\DeclareMathAlphabet{\mathscr}{OT1}{pzc}{m}{it}

\begin{document} 
\title{\textbf{Generic existence of multiplicity-$1$ minmax minimal hypersurfaces via Allen--Cahn}}
\author{Costante Bellettini\thanks{Research partially supported by the EPSRC (grant EP/S005641/1).}\\
University College London}
\date{}

\maketitle

\begin{abstract}
In Guaraco's work \cite{Gua} a new proof was given of the existence of a closed minimal hypersurface in a compact Riemannian manifold $N^{n+1}$ with $n\geq 2$. This was achieved by employing an Allen--Cahn approximation scheme and a one-parameter minmax for the Allen--Cahn energy (relying on works by Hutchinson, Tonegawa, Wickramasekera to pass to the limit as the Allen-Cahn parameter tends to $0$). The minimal hypersurface obtained may a priori carry a locally constant integer multiplicity. Here we consider a minmax construction that is a modification of the one in \cite{Gua}, by allowing an initial freedom on the choice of the valley points between which the mountain pass construction is carried out, and then optimising over said choice. We prove that, when $2\leq n\leq 6$ and the metric is bumpy, this minmax leads to a (smooth closed) minimal hypersurface with multiplicity $1$. (When $n=2$ this conclusion also follows from Chodosh--Mantoulidis's work \cite{ChoMan}.) As immediate corollary we obtain that every compact Riemannian manifold of dimension $n+1$, $2\leq n\leq 6$, endowed with a bumpy metric, admits a two-sided smooth closed minimal hypersurface (this existence conclusion also follows from Zhou's result \cite{Zhou3} for minmax constructions via Almgren--Pitts theory). 
\end{abstract}

\section{Introduction}
Existence problems in Riemannian geometry have a long history and those concerned with stationary points of area (and related functionals) occupy a prominent position. A minmax approach introduced in the 70s by Almgren and Pitts has lead in the last decade to extraordinary developments in geometric analysis, in particular after the celebrated work \cite{MN} by Marques--Neves. More recently, an alternative minmax approach has been developed, based on the approximation of the area functional by the Allen--Cahn energy. In particular, the existence of a closed minimal hypersurface in an arbitrary compact Riemannian manifold $N^{n+1}$ ($n\geq 2$), originally proved in \cite{Alm}, \cite{Pit}, \cite{SSY}, \cite{SS}, has been achieved in Guaraco's work \cite{Gua}, relying on \cite{HT}, \cite{TonWic}, \cite{Ton}, \cite{Wic}. Both in the Almgren--Pitts and in the Allen--Cahn approach, the minimal hypersurface is obtained as an integral varifold, that turns out to be smooth away from a singular set of codimension $\geq 7$ thanks to the fundamental regularity/compactness theory available in each of the two settings. In general, in both approaches, the hypersurface may a priori carry multiplicity $>1$. 

Very recently, Zhou \cite{Zhou3} obtains multiplicity-$1$ and two-sidedness for minmax minimal hypersurfaces constructed via the Almgren--Pitts minmax when $2\leq n\leq 6$ and $N$ is endowed with a bumpy (thus generic) metric (or, also, with a metric with positive Ricci curvature). The result applies to multi-parameter minmax, and confirms a well-known conjecture of Marques--Neves \cite[1.2]{MarNev2}. Previous progress in this direction had been made in \cite{MarNev}.

A natural counterpart to Marques--Neves's conjecture is expected to be true for minmax constructions via Allen--Cahn. For $n=2$ it follows from the work of Chodosh--Mantoulidis \cite{ChoMan} (valid for arbitrary solutions with finite Morse index, not necessarily minmax solutions) that the minimal surface obtained by a (one- or multi-parameter) Allen-Cahn minmax is two-sided with multiplicity $1$ in the case of bumpy metrics (and in the case of metrics with positive Ricci). For the one-parameter minmax in \cite{Gua}, it is proved by the author in \cite{B1} that if the metric has positive Ricci curvature then the multiplicity is $1$.

We point out that \cite{Zhou3} and \cite{ChoMan} obtain from their multi-parameter multiplicity-$1$ results (combined with the Weyl Laws available respectively for Almgren--Pitts and Allen--Cahn minmax constructions \cite{LMN}, \cite{GuaGas}) the existence of infinitely many minimal hypersurfaces. In other words, they establish, under their respective assumptions, the validity of (versions of) the well-known Yau's conjecture, which is established by other methods and for arbitrary Riemannian metrics with $2\leq n\leq 6$ by the combined efforts of Marques--Neves \cite{MarNev3} and Song \cite{Song} (for generic metrics with $n\geq 7$, see \cite{YLi}). The relevance of multiplicity-$1$ results is shared by all minmax constructions (and not only): a minmax procedure developed by Rivi\`{e}re \cite{Riv} for $2$-dimensional surfaces in arbitrary codimension also faces the same issue, resolved by Pigati--Rivi\`{e}re \cite{PigRiv}.

We obtain here a multiplicity-$1$ result in the case in which $2\leq n\leq 6$ and $N$ has a bumpy (thus generic) metric, for a one-parameter minmax construction (via Allen--Cahn energy) that is a modification of the one set up in \cite{Gua}, and of which we now give a brief overwiew, with details given in Section \ref{minmax_setup}. (The minmax construction itself can be performed for an arbitrary Riemannian metric and in any dimension, while the multiplicity-$1$ conclusion exploits the metric and dimensional restrictions.)

For each $\eps$ (the parameter of the Allen--Cahn energy $\Ece$), instead of using the admissible class of paths in $W^{1,2}(N)$ that join the constant $-1$ to the constant $+1$ (as in \cite{Gua}), we consider all continuous paths in $W^{1,2}(N)$ that connect two distinct strictly stable critical points $v^1_{\eps}, v^2_{\eps}$ of the Allen--Cahn energy. (Note that the constants $-1$ and $+1$ are possible choices of strictly stable critical points.) For each $\eps$ and for any such $v^1_{\eps}\neq v^2_{\eps}$, the minmax produces by a standard mountain pass lemma an Allen--Cahn critical point $u_{(v^1_{\eps}, v^2_{\eps})}$ with Morse index at most $1$, with Allen--Cahn energy $\Ece(u_{(v^1_{\eps}, v^2_{\eps})})$ realizing the minmax value. Then we consider (for each $\eps$)
$$\inf_{v^1_{\eps}, v^2_{\eps}}\Ece(u_{(v^1_{\eps}, v^2_{\eps})}),$$
where $v^1_{\eps}\neq v^2_{\eps}$ vary among all possible strictly stable critical points of $\Ece$. We note that this infimum is achieved by the energy of a critical point $u_{\eps}$, that has Morse index at most $1$. Next we let $\eps\to 0$, and consider any (subsequential) varifold limit $V$ of the family $V^{u_{\eps}}$, the varifolds associated to $u_{\eps}$. It then follows that $V\neq 0$ and (as in \cite{Gua}, using \cite{HT}, \cite{Ton} \cite{TonWic}, \cite{Wic}) $\spt{V}$ is smoothly embedded except possibly for a singular set $\text{sing}\,V$ of dimension $\leq n-7$, and $\spt{V}\setminus \text{sing}\,V$ carries locally constant integer multiplicity. We then establish:

\begin{thm}
\label{thm:mult_one_bumpy}
Let $2\leq n\leq 6$ and let $N$ be a compact manifold of dimension $n+1$. There exists a set of Riemannian metrics on $N$ that is generic in the sense of Baire category (specifically, the bumpy metrics of \cite{White1}) such that any varifold $V$ obtained by the minmax in Section \ref{minmax_setup} is the multiplicity-$1$ varifold associated to a smooth (embedded) closed minimal hypersurface $M$.
\end{thm}

When $n=2$, Theorem \ref{thm:mult_one_bumpy} follows from the more general result in \cite{ChoMan} (that applies under a uniform bound on the Morse index and on the energy). The multiplicity-$1$ information in Theorem \ref{thm:mult_one_bumpy} implies immediately that $M$ in Theorem \ref{thm:mult_one_bumpy} is the (common) boundary of two disjoint open sets and therefore it is a two-sided hypersurface; moreover, it has Morse index $1$ (by the bumpy metric assumption). In particular, Theorem \ref{thm:mult_one_bumpy} provides a proof of the following geometric result (that also follows from \cite{Zhou3}, which employs an Almgren--Pitts framework).

\begin{cor}
\label{cor:two_sided_bumpy}
Let $2\leq n\leq 6$. In any compact Riemannian manifold of dimension $n+1$ endowed with a bumpy (thus generic) metric there exists a (smoothly embedded) closed, two-sided minimal hypersurface, with Morse index $1$.
\end{cor}

\begin{oss}
It also follows, under the assumptions of Theorem \ref{thm:mult_one_bumpy}, that the critical points $u_{\eps}$ employed in the construction of $V$ can themselves be obtained as mountain pass solutions, for the class of admissible paths that join two (suitably chosen) strictly stable critical points; moreover, there is an optimal path, i.e.~one for which the maximum of $\Ece$ is achieved at $u_{\eps}$.
\end{oss}

In Section \ref{minmax_setup} we describe the minmax construction. In Section \ref{estimate} we identify the key estimate (Proposition \ref{Prop:main}) to which Theorem \ref{thm:mult_one_bumpy} can be reduced. The proof of this estimate will be given in Section \ref{proof_main}, after some preliminary work in Section \ref{preliminaries}. Once the minmax in Section \ref{minmax_setup} is in place, the proof of Theorem \ref{thm:mult_one_bumpy} is a variant of the one developed in \cite{B1}, in which positiveness of the Ricci curvature is assumed for compact Riemannian manifolds of dimension $3$ or higher. In the present setting, we rely on the dimensional restriction in order to have the smoothness of the support of the minmax varifold and the genericity of metrics for which every smooth minimal hypersurface admits no non-trivial Jacobi fields (these are the so-called bumpy metrics, proved to be generic by White in \cite{White1}). Moreover, we rely on the recent result \cite{GMN} by Guaraco--Marques--Neves, which implies that the orientable double cover of the (smooth) minmax minimal hypersurface, onto which index-$1$ Allen--Cahn solutions accumulate, cannot be strictly stable (and hence it is unstable under a bumpy metric assumption). With \cite{GMN} in mind, the statement of Theorem \ref{thm:mult_one_bumpy} could be viewed as an Allen--Cahn counterpart of the result in \cite[Section 1.6]{MarNev} on one-parameter Almgren--Pitts minmax.

\section{The minmax construction}
\label{minmax_setup}

Let $N$ be a compact Riemannian manifold of dimension $n+1$, $n\geq 2$. For $\eps\in(0,1)$ consider the Allen--Cahn energy 
$$\Ece(u)=\frac{1}{2\sigma}\int_N \eps\frac{|\nabla u|^2}{2} + \frac{W(u)}{\eps}$$
on the Hilbert space $W^{1,2}(N)$, where $W$ is a $C^2$ double well potential with non-degenerate minima at $\pm 1$, for example $W(x)=\frac{(1-x^2)^2}{4}$, suitably modified (for technical reasons) outside $[-2,2]$--- a typical choice, that we also follow, is to impose quadratic growth to $\infty$ (others are also possible), and $\sigma$ is a normalization constant, $\sigma=\int_{-1}^1 \sqrt{W(t)/2}\,dt$. We recall that the Euler--Lagrange equation for $\Ece$ is the semi-linear elliptic PDE $\eps \Delta u - \frac{W'(u)}{\eps} =0$ (where $\Delta$ denotes the Laplace--Beltrami operator on $N$), and that the second variation of $\Ece$ at $u$ is given by the quadratic form $\Ece''(u)(\phi,\phi)=\int_N \eps |\nabla \phi|^2 + \frac{W''(u)}{\eps}\phi^2$ for $\phi\in C^\infty(N)$. Stability amounts to $\Ece''(u)(\phi,\phi)\geq 0$ for all $\phi$, while strict stability means $\Ece''(u)(\phi,\phi)>0$ for all $\phi\neq 0$.

\medskip

For any continuous path in $W^{1,2}(N)$ that starts at the constant $-1$ and ends at the constant $+1$ there exists $\alpha>0$ such that the maximum of $\Ece$ on the path is $\geq \alpha$. This is proved in \cite{Gua}. (In fact the constant $\alpha$ is independent of $\eps$.) The constants $\pm 1$ are strictly stable critical points of $\Ece$ (they are also the only global minimizers and $\Ece(\pm 1) =0$).

For any strictly stable critical point $v$ of $\Ece$ we have a neighbourhood of $v$ in $W^{1,2}(N)$ in which $v$ achieves the (strict) minimum of $\Ece$. This follows from the Morse--Palais lemma (see e.g.~\cite[Lemma 7.3.1]{Jost}) for non-degenerate critical points of smooth functionals on Banach spaces. Consider a continuous path $\gamma:[a,b]\to W^{1,2}(N)$ such that $\gamma(a)$ and $\gamma(b)$ are strictly stable critical points of $\Ece$. Then the maximum of $\Ece$ on $\gamma$ has to be strictly greater than $\max\{\gamma(a), \gamma(b)\}+\delta$, where $\delta>0$ is independent of $\gamma$.

For any pair of strictly stable critical points $v_1$, $v_2$ of $\Ece$ we define the class of paths $\Gamma_{v_1, v_2}$ to be the collection of continuous paths $\gamma:[a,b]\to W^{1,2}(N)$ with endpoints $\gamma(a), \gamma(b)$ respectively equal to $v_1$ and $v_2$. The previous considerations guarantee that we have a ``mountain pass condition'', i.e.~there is a value $C_{v_1, v_2}>\max\{\Ece(v_1),\Ece(v_2)\}$ and such that for every $\gamma\in\Gamma_{v_1, v_2}$ we have $\max_{t\in[a,b]} \Ece(\gamma_t)\geq C_{v_1,v_2}$. Moreover, the Palais--Smale condition is satisfied by $\Ece$ (by the argument in \cite[Proposition 4.4 (ii)]{Gua}). This allows the use of a standard mountain pass theorem and yields the existence of a minmax solution to $\Ece'=0$ with Morse index $\leq 1$ and whose Allen--Cahn energy realises the minmax value $\min_{\gamma\in \Gamma_{v_1,v_2}} \max_{t\in[a,b]} \Ece(\gamma_t)$. (For example, see \cite{Struwe}.)

\medskip

Given $\eps\in(0,1)$ we will denote by $\Se$ the collection of strictly stable critical points of $\Ece$ on $N$. We will consider, for any $v_{\eps}^1, v_{\eps}^2 \in \Se$, $v_{\eps}^1\neq  v_{\eps}^2$, the admissible class of paths $\Gamma_{v_{\eps}^1, v_{\eps}^2}$: the mountain pass theorem yields (as just described) a critical point $u_{(v_{\eps}^1, v_{\eps}^2)}$ (of $\Ece$) with Morse index $\leq 1$ and with 
$$\Ece(u_{(v_{\eps}^1, v_{\eps}^2)})=\min_{\gamma\in \Gamma_{v_{\eps}^1, v_{\eps}^2}}\,\, \max_{t\in[a,b]} \Ece(\gamma_t).$$
We now \textit{``optimise''} the choice of the valley points $v_{\eps}^1, v_{\eps}^2$: as $v_{\eps}^1 \neq v_{\eps}^2$ vary in $\Se$, we consider the ``infimum of the minmax values'', namely 
$$\inf\limits_{\tiny{\begin{array}{ccc}
         v_{\eps}^1, v_{\eps}^2 \in \Se\\
         v_{\eps}^1\neq v_{\eps}^2
        \end{array}}} \Ece(u_{(v_{\eps}^1, v_{\eps}^2)})=\inf\limits_{\tiny{\begin{array}{ccc}
         v_{\eps}^1, v_{\eps}^2 \in \Se\\
         v_{\eps}^1\neq v_{\eps}^2
        \end{array}}}\min_{\gamma\in \Gamma_{v_{\eps}^1, v_{\eps}^2}} \,\,\max_{t\in[a,b]} \Ece(\gamma_t);$$
we will now check that there exists a critical point $u_{\eps}$ of $\Ece$ such that 
\be
\label{eq:inf_u_eps}
\Ece(u_{\eps})=\inf\limits_{\tiny{\begin{array}{ccc}
         v_{\eps}^1, v_{\eps}^2 \in \Se\\
         v_{\eps}^1\neq v_{\eps}^2
        \end{array}}} \Ece(u_{(v_{\eps}^1, v_{\eps}^2)}).
\ee
Indeed, taking an infimizing sequence $(v_{\eps}^1, v_{\eps}^2)_\ell$ for $\ell \to \infty$, we have a uniform bound on $\Ece(u_{{(v_{\eps}^1, v_{\eps}^2)}_\ell})$ along the sequence and thus a uniform $W^{1,2}$ bound on $u_{{(v_{\eps}^1, v_{\eps}^2)}_\ell}$; we first extract a weak $W^{1,2}$-limit of $u_{{(v_{\eps}^1, v_{\eps}^2)}_\ell}$, as $\ell\to \infty$, that we denote by $u_{\eps}$; by passing to the limit in the weak version of $\Ece'=0$ we obtain that $u_{\eps}$ is a weak solution to the Allen--Cahn equation; then we show that the convergence to $u_{\eps}$ is strong in $W^{1,2}$ by the stationarity assumption $\Ece'\left(u_{{(v_{\eps}^1, v_{\eps}^2)}_\ell}\right)=0$ (the computation is again the same as in \cite[Proposition 4.4 (ii)]{Gua}). Elliptic theory guarantees smoothness of $u_{\eps}$ and the fact that it solves $\Ece'(u_{\eps})=0$ in the strong sense. Moreover, the minimizing sequence $u_{{(v_{\eps}^1, v_{\eps}^2)}_\ell}$ converges (as $\ell\to \infty$) to $u_{\eps}$ in $C^k(N)$ for any $k\in \N$, by elliptic estimates, and $\Ece(u_{{(v_{\eps}^1, v_{\eps}^2)}_\ell})\to \Ece(u_{\eps})$. 

We notice that $u_{\eps}$ has Morse index $\leq 1$; this follows from the Rayleigh quotient characterisation of the eigenvalues (see e.g.~\cite[(3.21)]{Hie}), from the strong convergence of $u_{{(v_{\eps}^1, v_{\eps}^2)}_\ell}$ to $u_{\eps}$ and from the fact that each $u_{{(v_{\eps}^1, v_{\eps}^2)}_\ell}$ has Morse index $\leq 1$ for each $\ell$. (It suffices to prove that if $f_\ell\to f_\infty$ in $C^2$ and $\ca{E'}{\eps}(f_\ell)=0$ for $\ell\in \N \cup \{\infty\}$, denoting by $\lambda_p^{f_\ell}$ the $p$-th eigenvalue of the Jacobi operator associated to $\Ece$ and $f_\ell$, for $\ell\in \N \cup \{\infty\}$, then $\limsup_{\ell\to \infty} \lambda_p^{f_\ell} \leq \lambda_p^{f_\infty}$.)

\medskip

\noindent \textit{Associated varifolds}. In order to produce candidate minimal hypersurfaces (i.e.~stationary integral varifolds) we follow the construction in \cite{HT}. Given a smooth function $u:N\to\R$ we let $w = \Phi(u)$, where $\Phi(s)=\int_0^s \sqrt{W(t)/2}\,dt$. We denote by $\sigma$ the normalization constant $\int_{-1}^1 \sqrt{W(t)/2}\,dt$. Then we define the $n$-varifolds 
$$V^{u}(A) = \frac{1}{\sigma}\int_{-\infty}^\infty V_{\{w=t\}}(A) \,dt,$$
where $A\subset G_n(N)$ and $V_{\{w=t\}}$ denotes, for a.e.~$t$, the varifold of integration on the (smooth) level set ${\{w=t\}}$. If, for $\eps=\varepsilon_j \to 0^+$ the funtions $u_{\eps}$ are critical points of $\Ece$ and $\Ece(u_{\eps})$ is uniformly bounded, then the analysis in \cite{HT} gives that $V^{u_{\eps}}$ converge subsequentially, as $\eps \to 0$, to an integral $n$-varifold $V$ with vanishing first variation. Moreover $\Ece(u_{\eps})\to \|V\|(N)$, the total mass of $V$. 

\medskip

\noindent \textit{Upper and lower energy bounds.} We have positive upper and lower bounds on $\Ece(u_{\eps})$ as $\eps\to 0$, for the critical points $u_{\eps}$ constructed in (\ref{eq:inf_u_eps}). For the upper bound this follows from the upper bound (independent of $\eps$) obtained in \cite{Gua} for $\limsup_{\eps\to 0} c_{\eps}$, where $c_{\eps}$ is the minmax value obtained by employing the class of paths $\Gamma_{-1,+1}$, together with the infimum characterisation of $u_{\eps}$ in (\ref{eq:inf_u_eps}). The lower bound follows from the lower bound obtained in \cite{Gua} for $\liminf_{\eps\to 0} c_{\eps}$, together with the following observation. 
If $\Ece(v_{\eps})\to 0$ for a sequence of critical points ($\Ece'(v_{\eps})=0$) with $\eps=\varepsilon_j\to 0^+$, then for all sufficiently large $j$ we have $v_{\eps}\equiv -1$ or $v_{\eps}\equiv +1$.
This is proved by a blow up argument, as in \cite{BW3}. (It suffices to prove that if $\Ece(v_{\eps})\to 0$ then $\{v_{\eps}=0\}=\emptyset$ for sufficiently small $\eps$; then the maximum principle gives the conclusion. Arguing by contradiction, we let $x_j\in \{v_{{\varepsilon}_j}=0\}$ and define $\tilde{v}_j(\cdot)=v_{{\varepsilon}_j}\left(\frac{\cdot - x_j}{{\varepsilon}_j}\right)$. Sending $j\to \infty$ we obtain an entire solution $v:\R^{n+1}\to \R$ to $\ca{E'}{1}(v)=0$ with $v(0)=0$ and $\ca{E}{1}(v)=0$, contradiction.) This observation equivalently says that there exist $\eps_0>0$ and $C>0$ such that if $v_{\eps}$ is a critical point of $\Ece$ for $\eps\leq \eps_0$ and $v_{\eps}\not\equiv -1$, $v_{\eps}\not\equiv -1$, then $\Ece(v_{\eps})\geq C$. For the construction above (see the discussion preceding (\ref{eq:inf_u_eps}), if $\eps\leq \eps_0$ and at least one between $v_{\eps}^1, v_{\eps}^2$ is not $\pm 1$, then we have $\Ece(u_{(v_{\eps}^1, v_{\eps}^2)})\geq \max\{\Ece(v_{\eps}^1), \Ece(v_{\eps}^2)\}\geq C>0$. If $v_{\eps}^1, v_{\eps}^2$ are the constants $-1, +1$, on the other hand, then by \cite{Gua} for all sufficiently small $\eps$ we have $\Ece(u_{(-1, +1)})\geq \frac{1}{2}\liminf_{\eps\to 0} c_{\eps}>0$. Therefore a positive lower bound for $\liminf_{\eps\to 0} \Ece(u_{(v_{\eps}^1, v_{\eps}^2)})$ exists.

\medskip

\noindent \textit{Varifold limit and regularity.} Following \cite{Gua} we consider any (subsequential) varifold limit of $V^{u_{\eps}}$ as $\eps \to 0$; for $n\geq 2$ we get that $\spt{V}$ is a smoothly embedded minimal hypersurface except possibly for a closed singular set of dimension $\leq n-7$. This follows upon noticing that the uniform bound on the Morse index of $u_{\eps}$ allows to reduce \textit{locally} in $N$ to the case in which $u_{\eps}$ are stable so that the regularity results in \cite{TonWic}, \cite{Wic} apply. Note that $V\neq 0$ by the lower energy bound on $u_{\eps}$.
In other words, 
\be
\label{eq:V_M_j}
V=\sum_{j=1}^K q_j |M_j|,
\ee
with $q_j\in \N$ and $M_j$ minimal and smoothly embedded
away from a set of dimension $\leq n-7$ ($|M_j|$ denotes the multiplicity-$1$ varifold of integration on $M_j$).

\begin{oss}
The observation that $u_{\eps}$ has Morse index $\leq 1$ simplifies the exposition, however an alternative way to carry out the construction would be to consider a diagonal sequence of $u_{(v^1_{{\eps}}, v_{\eps}^2)}$ as $\eps\to 0$ such that the varifolds $V^{u_{(v^1_{{\eps}}, v_{\eps}^2)}}$ converge to the same limit $V$ as the varifolds $V^{u_{\eps}}$. Then the regularity of $V$ could be obtained from the knowledge that the Morse index of $u_{(v^1_{{\eps}}, v_{\eps}^2)}$ is $\leq 1$.
\end{oss}

\medskip

The previous construction can be carried out for any $n\geq 2$ and any Riemannian metric on $N$. In the case $n\leq 6$ that we will be intersted in, all the $M_j$'s obtained in (\ref{eq:V_M_j}) are completely smooth. The scope, in the remainder of this work, is to prove that if $2\leq n\leq 6$ and the metric is bumpy, then all the multiplicities $q_j$ in (\ref{eq:V_M_j}) must be equal to $1$. This will establish Theorem \ref{thm:mult_one_bumpy}.

\begin{oss}
\label{oss:index_u_eps}
It is also true (see Remark \ref{oss:optimalpath}) that if $2\leq n\leq 6$ and the metric is bumpy, then $u_{\eps}$ itself can be obtained, for all sufficiently small $\eps$, by a minmax in the class $\Gamma_{v_{\eps}^1, v_{\eps}^2}$ for a suitable choice of $v_{\eps}^1\neq v_{\eps}^2 \in \Se$, in particular
$$\Ece(u_{\eps})=\min\limits_{\tiny{\begin{array}{ccc}
         v_{\eps}^1, v_{\eps}^2 \in \Se\\
         v_{\eps}^1\neq v_{\eps}^2
        \end{array}}} \Ece(u_{(v_{\eps}^1, v_{\eps}^2)})=\min\limits_{\tiny{\begin{array}{ccc}
         v_{\eps}^1, v_{\eps}^2 \in \Se\\
         v_{\eps}^1\neq v_{\eps}^2
        \end{array}}} \min_{\gamma\in \Gamma_{v_{\eps}^1, v_{\eps}^2}} \,\,\max_{t\in[a,b]} \Ece(\gamma_t).$$
\end{oss}
\section{Key estimate and proof of Theorem \ref{thm:mult_one_bumpy}}
\label{estimate}

Theorem \ref{thm:mult_one_bumpy} will follow mainly from the following key estimate (in which the dimensional restriction is absent in view of the fact that $M$ is assumed to be smoothly embedded).

\begin{Prop}
\label{Prop:main}
Let $N$ be a compact Riemannian manifold of dimension $n+1$ with $n\geq 2$ and let $M\subset N$ be a smoothly embedded, closed minimal hypersurface, whose oriented double cover is unstable. There exists $\varsigma>0$ (depending only on $M\subset N$) and $\eps_0>0$ such that for every $\eps<\eps_0$ there exist a stable Allen--Cahn critical point $v_{\eps}$ ($\ca{E'}{{\eps}}(v_{\eps})=0$, $\ca{E''}{{\eps}}(v_{\eps})\geq 0$) and a continuous path $\gamma:[a,b]\to W^{1,2}(N)$ with $\gamma(a)=-1$ and $\gamma(b)=v_{\eps}$ such that
$$\max_{t\in[a,b]}\Ece(\gamma(t))\leq 2\mathcal{H}^n(M)-\varsigma.$$
\end{Prop}

\noindent Additionally we will need the following lemma, whose proof (see Appendix \ref{double_cover}) follows from Lemma \ref{lem:stable_implies_strictly} and from \cite{GMN}.

\begin{lem}
\label{lem:stable_implies_strictly_2}
Let $N$ be a compact Riemannian manifold of dimension $n+1$ with $2\leq n\leq 6$, endowed with a bumpy metric. For any $K>0$ there exists $\epsilon_0>0$ such that if $\eps\in(0,\epsilon_0)$ and $v_{{\eps}}:N\to \R$ is a stable critical point of $\Ece$ with $\Ece(v_{\eps})\leq K$, then $v_{{\eps}}$ is strictly stable.
\end{lem}

\begin{proof}[proof of Theorem \ref{thm:mult_one_bumpy} assuming Proposition \ref{Prop:main} and Lemma \ref{lem:stable_implies_strictly_2}]
We recall that the varifold $V=\sum_j q_j |M_j|$ obtained in Section \ref{minmax_setup} satisfies $\|V\|(N)=\lim \Ece(u_{\eps})$ for $\eps={\eps}_j\to 0^+$. It follows from \cite{GMN} that if for some $j_0$ the oriented double cover of $M_{j_0}$ is stable, then $q_{j_0}=1$. Indeed, since each $M_{j_0}$ is smoothly embedded we can choose a tubular neighbourhood $T_{j_0}$ of it such that $\spt{V} \cap T_{j_0}=M_{j_0}$. The oriented double cover of $M_{j_0}$ is strictly stable by the bumpy metric assumption, so using \cite{GMN} in $T_{j_0}$ we conclude that the associated varifolds $V^{u_{\eps}}\res T_{j_0}$ converge with multiplicity $1$. Since $V^{u_{\eps}}\to V=\sum_j q_j |M_j|$ on $N$, we conclude that $q_{j_0}=1$. On the other hand, Proposition \ref{Prop:main} and Lemma \ref{lem:stable_implies_strictly_2} imply that if $j_0$ is such that $M_{j_0}$ has unstable double cover, then $q_{j_0}$ must be $1$. Indeed, if that were not the case, we could choose $M=M_{j_0}$ in Proposition \ref{Prop:main}, obtaining the existence of a path $\gamma\in \Gamma_{-1,v_{\eps}}$, for some strictly stable $v_{\eps}$, such that $\max_{t\in[a,b]}\Ece(\gamma(t))\leq 2\mathcal{H}^n(M_{j_0})-\varsigma$; a fortiori, 
$$\min_{\gamma \in  \Gamma_{-1,v_{\eps}}} \max_{t\in[a,b]}\Ece(\gamma(t))\leq 2\mathcal{H}^n(M_{j_0})-\varsigma$$
for $\eps<\min\{\epsilon_0, \varepsilon_0\}$, and 
$$\Ece(u_{\eps})=\inf\limits_{\tiny{\begin{array}{ccc}
         v_{\eps}^1, v_{\eps}^2 \in \Se\\
         v_{\eps}^1\neq v_{\eps}^2
        \end{array}}} \,\min_{\gamma \in  \Gamma_{v_{\eps}^1, v_{\eps}^2}} \max_{t\in[a,b]}\Ece(\gamma(t))\leq 2\mathcal{H}^n(M_{j_0})-\varsigma$$
for $\eps<\min\{\epsilon_0, \varepsilon_0\}$. Since $\|V\|(N)=\lim_{\eps\to 0} \Ece(u_{\eps})$ (where $\eps={\eps}_j$ is any subsequence that led to (\ref{eq:V_M_j})) and $\|V\|(N)\geq q_{j_0} \mathcal{H}^n(M_{j_0})$ we conclude that $q_{j_0}=1$.
\end{proof}

\section{Preliminary results}
\label{preliminaries}

\subsection{Truncated $1$-dimensional Allen--Cahn solutions}
\label{truncations}
We denote by $\Het(r)$ the monotonically increasing solution to $u''-W'(u)=0$ such that $\lim_{r\to \pm\infty} \Het(r) = \pm 1$, with $\Het(0)=0$. (For the standard potential $\frac{(1-x^2)^2}{4}$, we have $\Het(r)=\tanh\left(\frac{r}{\sqrt{2}}\right)$.) Then the functions $\Het(-r)$ and $\Het(\pm r+z)$ also solve $u''-W'(u)=0$ (for any $z\in \R$). The rescaled function $\He(r)=\Het\left(\frac{r}{\eps}\right)$ solves ${\eps} u''-\frac{W'(u)}{\eps}=0$.

We will make use of truncated versions of $\He$
(see \cite{ChoMan}, \cite{WaWe}, \cite{B1} for details on this truncation): for $\Lambda=3|\log\eps|$ define

$$\OHet(r) = \chi(\Lambda^{-1} r -1)\Het(r) \pm (1-\chi(\Lambda^{-1}|r|-1)),$$
where $\pm$ is chosen respectively on $r>0$, $r<0$ and $\chi$ is a smooth bump function that is $+1$ on $(-1,1)$ and has support equal to $[-2,2]$. With this definition, $\OHet=\Het$ on $(-\Lambda, \Lambda)$, $\OHet=-1$ on $(-\infty, -2\Lambda]$, $\OHet=+1$ on $[2\Lambda,\infty)$. Moreover $\OHet$ solves  $\|\OHet''-W'(\OHet)\|_{C^2(\R)} \leq C \eps^3$, for $C>0$ independent of $\eps$. (Note also that $\OHet''-W'(\OHet)=0$ away from $(-2\Lambda, -\Lambda) \cup (\Lambda,2\Lambda)$.)

For $\eps<1$, we rescale these truncated solutions and let $\OHet^{\eps}(\cdot)=\OHet\left(\frac{\cdot}{\eps}\right)$. Note that $\OHet^{\eps}$ solves $\|\eps\OHet''-\frac{W'(\OHet)}{\eps}\|_{C^2(\R)} \leq C \eps^2$ and $\eps\OHet''-\frac{W'(\OHet)}{\eps}=0$ on $(-\eps\Lambda, \eps\Lambda)$, $\OHet^{\eps}=+1$ on $(2\eps\Lambda, \infty)$, $\OHet^{\eps}=-1$ on $(-\infty,-2\eps\Lambda)$.

Using these facts and recalling that $\Ece(\Het_{\eps})=1$ we get $\Ece(\OHet^{\eps})=1+O({\eps}^2)$. (The function $O({\eps}^2)$ is bounded by $C{\eps}^2$ for all $\eps$ sufficiently small, with $C$ independent of $\eps$.)

\medskip

We define the function $\Psi:\R\to \R$
\be
\Psi(r)=\left\{\begin{array}{ccc}
\OHet^{\eps}(r+2\eps \Lambda) & r\leq 0\\
\OHet^{\eps}(-r+2\eps \Lambda) & r>0
        \end{array}\right. .
\ee
This function is smooth thanks to the fact that all derivatives of $\OHet^{\eps}$ vanish at $\pm 2\eps \Lambda$. Moreover let $\Psi_t$ denote the following family of functions, with $\Psi_0=\Psi$ and $t\in [0,\infty)$:
\be
\label{eq:family2}
\Psi_t(r):= \left\{ \begin{array}{ccc}
                      \OHet^{\eps}(r+2\eps \Lambda-t) & r\leq 0 \\
                      \OHet^{\eps}(-r+2\eps \Lambda-t) & r> 0 
                     \end{array}
\right. .
\ee
We have $\Psi_0=\Psi$. Moreover, $\Psi_{t}\equiv -1$ for $t\geq 4\eps\Lambda$. For $t\in(0,4\eps\Lambda)$ the function $\Psi_t$ is equal to $-1$ for $r$ such that $|r|\geq 4\eps\Lambda-t$.
The functions $\Psi_t$ form a family of even, Lipschitz functions. The energy $\Ece(\Psi_t)$ is decreasing in $t$: indeed, the energy contribution of the ``tails'' is zero and we have $\Ece(\Psi_t)=\Ece(\Psi)-\frac{1}{2\sigma}\int_{-t}^t \eps |\Psi'|^2 + \frac{W(\Psi)}{\eps}$.

\subsection{Large unstable region}
\label{unstable_region}

Let $N$ be a Riemannian manifold of dimension $n+1$ with $n\geq 2$. 
Let $M\subset N$ be a smooth closed minimal hypersurface such that its oriented double cover $\tilde{M}$ is unstable. We denote by $\iota: \tilde{M}\to N$ the minimal smooth immersion induced by the projection of $\tilde{M}$ onto $M$. Let $\nu$ be a choice (on $\tilde{M}$) of unit normal to the immersion $\iota$.

Let $\om>0$ be the semi-width of a well-defined tubular neighbourhood of $M$ in $N$, with $\om<\text{inj}(N)$. Define the map $\tilde{M} \times [0,\om)\to N$ given by $$(p,s) \to \text{exp}_{\iota(p)}(s \nu(p)),$$
where $\text{exp}_{\iota(p)}$ denotes the exponential map at $\iota(p)$ (from a ball in the tangent to $N$ at $\iota(p)$ to a geodesic ball in $N$ centred at $\iota(p)$). 
This is a bijective diffeomorsphism on $\tilde{M} \times (0,\omega)$ (the map is $2-1$ on $\tilde{M} \times \{0\}$). We will endow $\tilde{M} \times [0,\omega)$ with the pull-back metric from $N$. This metric is of the form $\mathscr{g}_s + ds^2$, where $\mathscr{g}_s$ is a Riemannian metric on $\tilde{M} \times \{s\}$. In the following, $\tilde{M}$ will be implicitly assumed to be a Riemannian $n$-dimensional manifold with the metric $\mathscr{g}_0$. By abuse of notation we will write $\iota + s \nu$ in place of $\text{exp}_{\iota(p)}(s \nu(p))$ and also $\iota + \varphi \nu$ in place of $\text{exp}_{\iota(p)}(\varphi(p) \nu(p))$, where $0\leq \varphi<\om$ is a smooth function on $\tilde{M}$.  

The quotient $(\tilde{M}\times [0,\om))\diagup\sim$, where $(p_1, s_1)\sim (p_2, s_2)$ if and only if $\iota(p_1)=\iota(p_2)$ and $s_1=s_2=0$, is the (open) tubular neighbhourhood of $M$ of semi-width $\om$. For notational convenience we will denote it by $\quo$. Whenever we define functions on $\tilde{M}\times [0,\om)$ they will be always even with respect to $(p,0)\in \tilde{M}\times \{0\}$ so that they can be identified with functions in $\quo$. (Lifts of functions on $M$ are exactly even functions on $\tilde{M}$. By lift of $\rho:M\to \R$ we mean the function $\tilde{\rho}$ on $\tilde{M}$ defined by $\tilde{\rho}(p) = \rho(\iota(p))$.)

We will now consider deformations of $\iota$ with initial velocity dictated by a function $\varphi \in C^2(\tilde{M})$. For $\varphi \in C^2(\tilde{M})$, consider the one-parameter family of immersions $\iota_t:\tilde{M}\to N$ defined, for $t\in(-\delta_0,\delta_0)$ (for some $\delta_0\in(0,\frac{\om}{\text{max}\varphi})$), by $(p,t)\to \text{exp}_{\iota(p)}(t \varphi(p) \nu(p))$, for $p\in\tilde{M}$. The first variation of area at $t=0$ is $0$ because $M$ is minimal. The second variation of area at $t=0$ is given by 
\be
\label{eq:second_var}
\int_{\tilde{M}}|\nabla \varphi|^2  d{\Hc}^n-\int_{\tilde{M}}\varphi^2 (|A|^2 +\Rc{N}(\nu,\nu)) d{\Hc}^n,
\ee
where $A$ denotes the second fundamental form of $\iota$, $\nabla$ the gradient on $\tilde{M}$ (with respect to $\mathscr{g}_0$), $\Rc{N}$ the Ricci tensor of $N$ and $d\mathcal{H}^n$ is $d\text{vol}_{\mathscr{g}_0}$.

\begin{lem}[unstable region]
\label{lem:unstableregion}
There exist a geodesic ball $D\subset M$ with radius $R_0$, with $0<R_0<\text{inj}(N)$, and a function $\tilde{\phi} \in C^2_c(\tilde{M})$ with $\text{supp}\tilde{\phi}\subset \tilde{M}\setminus \overline{\iota^{-1}(D)}$ and $\tilde{\phi}\geq 0$, such that 
\begin{equation}
 \label{eq:unstableregion}
 \int_{\tilde{M}}|\nabla \tilde{\phi}|^2  d{\Hc}^n-\int_{\tilde{M}}\tilde{\phi}^2 (|A|^2 +\Rc{N}(\nu,\nu)) d{\Hc}^n < 0.
\end{equation}
\end{lem}

\begin{proof}
Let $\eta$ be the first eigenfunction of the Jacobi operator (on $\tilde{M}$), i.e.~$\Delta + (|A|^2+\Rc{N}(\nu.\nu))$, where $\Delta$ is the Laplace--Beltrami operator ($\eta$ can be normalized e.g.~to have unit $L^2$ norm). Then by standard theory $\eta$ is smooth and strictly positive on $\tilde{M}$ and $\int_{\tilde{M}}|\nabla \eta|^2  d{\Hc}^n-\int_{\tilde{M}}\eta^2 (|A|^2 +\Rc{N}(\nu,\nu)) d{\Hc}^n<0$. Pick an arbitrary $b\in M$ and let $\{b_1, b_2\}=\iota^{-1}(b)$. By standard capacity properties, given $\delta>0$ arbitrary there exists $\rho\in C^\infty_c(N)$, $0\leq \rho\leq 1$, that vanishes in a neighbourhood of $b$ and is identically one away from a (slightly larger) neighbourhood of $b$ and such that $\int_N |\nabla \rho|^2<\delta$. These properties of $\rho$ imply that for $\delta$ sufficiently small it is possible to replace $\eta$ by $\eta (1-\rho\circ \iota)$ (which is still non-negative) in the previous inequality and still obtain a negative result. Then we set $\tilde{\phi}=\eta (1-\rho\circ \iota)$. By construction there exists a geodesic ball $D\subset M$ (centred at $b$) contained in $\{1-\rho=0\}$.
\end{proof}

\begin{oss}[choice of $B$]
\label{oss:choiceB}
For $D$ obtained in Lemma \ref{lem:unstableregion}, we choose a (geodesic) ball $B$ in $M$ centred at $b$ and of radius $R<\frac{R_0}{2}$, where $R_0$ is the radius of $D$. The choices of $B$ and $\tilde{\phi}$ will be kept until the end. We will write $\tilde{D}=\iota^{-1}(D)$ and $\tilde{B}=\iota^{-1}(B)$.
\end{oss}

\begin{oss}
\label{oss:geom_counterpart}
The geometric counterpart of Lemma \ref{lem:unstableregion} is that the minimal immersion $\iota$ is unstable with respect to the area functional also if we restric to deformations that leave $\tilde{B}$ (and $B$) fixed, see Remark \ref{oss:choiceB2}.
\end{oss}

\subsection{Relevant immersions}
\label{immersions}

We have a natural unsigned Riemannian distance on $\tilde{M}\times [0,\om)$. In particular, the unsigned distance to $M$ of a point $(p,s) \in \tilde{M}\times [0,\om)$ is given by $s$. The unsigned distance to $M$ descends to the usual Riemannian distance to $M$ in $\quo$ and is a smooth function in $\quo$.

Let $\Pi$ denote the nearest point projection onto $M$, well-defined in $\quo$, $\Pi(p,s)=(p,0)$. For future purposes, we choose $\om$ suitably small so that if $x$ is in $\quo$, then $\left|\,\,|J\Pi|(x)-1\,\,\right|\leq 2 K'_A d(x,M)$ and $\left|\,\,\frac{1}{|J\Pi|(x)}-1\,\,\right|\leq 2K'_A d(x,M)$, where $d$ is the Riemannian distance and $K'_A$ is the maximum of the norm of the second fundamental form of $M$.

Given an immersion $\iota + \varphi \nu$, where $\varphi > 0$ is a (positive) smooth function on $\tilde{M}$, with $\varphi<{\om}$, the image of this immersion  is a two-sided embedded hypersurface with boundary, that we will denote by $M_{\varphi}$. We will always assume that the choice of normal to $M_{\varphi}$ is the one for which the scalar product with $\frac{\p}{\p s}$ is positive. 

\textit{Signed distance to $M_{\varphi}$.} We define, on $\left({\quo}\right) \setminus M$,
the following ``signed distance to $M_{\varphi}$''. Using the identification with $\tilde{M} \times (0,\omega)$, so that $M_{\varphi}$ is identified with the graph of $\varphi$ over $\tilde{M}$, we say that $(p,s)$ has negative distance to $M_{\varphi}$ if $s<\varphi(p)$ and positive distance to $M_{\varphi}$ if $s>\varphi(p)$. The modulus of the signed distance is the unsigned distance to $\text{graph}(\varphi)$ in $\tilde{M} \times (0,\omega)$ (recall that the latter is endowed with the Riemannian metric pulled back from $N$). If $(p,s) \in \text{graph}(\varphi)$ then the distance extends smoothly at $(p,s)$ with value $0$. (We do not define the unsigned distance on $M$.)

\begin{oss}[relevant immersions]
 \label{oss:def_immersions}
Recall the function $\tilde{\phi}$ given in Lemma \ref{lem:unstableregion}.
For sufficiently small constants $\tilde{c}_0>0$ and $\tilde{t}_0>0$ the following immersions are well-defined for all $c\in [0,\tilde{c}_0], t\in [0,\tilde{t}_0]$ (the constants $\tilde{c}_0, \tilde{t}_0$ need to be sufficiently small to ensure that we stay in $\tilde{M} \times [0,\omega)$):

$$\begin{array}{ccc}
   \tilde{M}  & \to & {\quo}\\
   p &\to & (p,c+t\tilde{\phi}(p))
  \end{array}
 \text{ or, equivalently, } \iota+(c+t\tilde{\phi})\nu:\tilde{M} \to N.$$
The image of this immersion is identified with
$$\text{graph}\left((c+t\tilde{\phi})\right)\subset \tilde{M} \times [0,\omega].$$
(The immersion is not necessarily even in $p$.)

Note that a suitably small choice of $\tilde{c}_0>0$ and $\tilde{t}_0>0$ additionally guarantees the following technically useful fact: let $0<c_1<\tilde{c}_0$, for every $c\in[c_1,\tilde{c}_0]$ and $t\in [0,\tilde{t}_0]$ there exists a tubular neighbourhood of size $c_1$ of the embedded separating hypersurface $\text{graph}\left((c+t\tilde{\phi}) \right)$. (We will use this with $c_1=6\eps|\log\eps|$ in Section \ref{proof_main}, \ref{peak} in particular.)

Let $B$ be the ball in Remark \ref{oss:choiceB}. We let $\tilde{M}_{B}=\tilde{M} \setminus \iota^{-1}(B)$ and $\iota_B=\left.\iota\right|_{\tilde{M}_{B}}$ (note that $\tilde{M}_{B}$ is a manifold with boundary). By abuse of notation we will write $\tilde{\phi}$ also to mean $\left.\tilde{\phi}\right|_{\tilde{M}_{B}}$. The immersions 
$$\begin{array}{ccc}
   \tilde{M}_{B} & \to & {\quo}\\
   p &\to & (p,c+t\tilde{\phi}(p))
  \end{array}\text{ or, equivalently, } \iota_B+(c+t\tilde{\phi})\nu:\tilde{M}_{B}\to N,$$ 
are well-defined for all $c\in [0,\tilde{c}_0], t\in [0,\tilde{t}_0]$. The image of this immersion with boundary is identified with 
$$\text{graph}\left(\left.(c+t\tilde{\phi})\right|_{\tilde{M}_{B}}\right).$$
\end{oss}

\begin{oss}
\label{oss:choiceB2}
By Lemma \ref{lem:unstableregion} and Remark \ref{oss:geom_counterpart} there exists $t_0\in [0,\tilde{t}_0]$ such that the area of the immersion $\iota+t \tilde{\phi}\nu: \tilde{M} \to N$ for $t\in[0,t_0]$ is strictly decreasing in $t$
(the second derivative of area at $t=0$ along the deformation $\iota+t \tilde{\phi}\nu$ is strictly negative). Note, moreover, that this deformation leaves $B$ fixed, so we equivalently have the following: the area of the immersion with boundary $\iota_B+ t\tilde{\phi}\nu: \tilde{M}_{B} \to N$ is strictly decreasing in $t$ for $t\in[0,t_0]$ (this is a deformation of the immersion with fixed boundary; the area at $t=0$ is $2{\Hc}^n(M)-2{\Hc}^n(B)$).
\end{oss}

\begin{lem}
\label{lem:unstable_addcnst_region}
Let $t_0$ be as in Remark \ref{oss:choiceB2}. There exist $c_0\in [0,\tilde{c}_0]$ and $\tau>0$ such that 

\begin{description}
 \item[(i)] for all $c\in[0,c_0]$ and for all $t\in [0,t_0]$ the area of the immersion (with boundary) $\iota_B+(c+t\tilde{\phi})\nu:\tilde{M}_{B} \to N$ is $\leq 2({\Hc}^n(M)-\frac{3}{4}{\Hc}^n(B))$;
 \item[(ii)] for all $c\in[0,c_0]$ the area of the immersion $\iota+(c+t_0\tilde{\phi})\nu: \tilde{M}\to N$ is $\leq 2{\Hc}^n(M)-\tau$.
\end{description}
\end{lem}

\begin{proof}
Let us prove that (i) holds for some $c'_0\in [0,\tilde{c}_0]$ (in place of $c_0$). Argue by contradiction: if not, then there exists $c_i\to 0$ and $t_i \in [0,t_0]$ such that the area of $\iota_B+(c_i+t_i\tilde{\phi})\nu$ is $\geq 2({\Hc}^n(M)-\frac{3}{4}{\Hc}^n(B))$ for all $i$. Upon extracting a subsequence we may assume $t_i \to t \in [0,t_0]$ and by continuity of the area we get that the area of $\iota_B+t\tilde{\phi}\nu$ is $\geq 2({\Hc}^n(M)-\frac{3}{4}{\Hc}^n(B))$. This is however in contradiction with Remark \ref{oss:choiceB2}, which says that this area is $\leq 2{\Hc}^n(M)-2{\Hc}^n(B)$.

Let us prove that (ii) holds for some $c''_0\in [0,\tilde{c}_0]$ (in place of $c_0$) and for some $\tau>0$. By Remark \ref{oss:choiceB2} the area of $\iota+t_0 \tilde{\phi}\nu$ is strictly smaller than $2{\Hc}^n(M)$. Denote by $2\tau$ the difference of the two areas. By continuity, there exists $c''_0>0$ such that for all $c\in [0,c''_0]$ the area of the immersion $\iota+ (c+t_0\tilde{\phi})\nu$ is smaller than $2{\Hc}^n(M)-\tau$.

Choosing $c_0=\min\{c'_0, c''_0\}$ concludes. 
\end{proof}

Since $\iota:\tilde{M}\to N$ is a closed smooth immersion we can find a constant $K_A>0$ such that: 

(i) the modulus of the second fundamental form along the immersion $\iota+ (c+t\tilde{\phi})\nu$ is $\leq K_A$ for all $c\in [0,c_0]$, $t\in [0,t_0]$;

(ii) let $M_{c,t}$ be the embedded hypersurface obtained as the image of $\iota+ (c+t\tilde{\phi})\nu$; the nearest point projection $\Pi_{c,t}$ from a tubular neighbourhood of $M_{c,t}$ with semi-width $c$, has the following bounds: $\left|\,\,|J\Pi_{c,t}|(x)-1\,\,\right|\leq K_A d(x,M_{c,t})$ and $\left|\,\,\frac{1}{|J\Pi_{c,t}|(x)}-1\,\,\right|\leq K_A d(x,M_{c,t})$, where $d$ is the Riemannian distance. To simplify notation, we also assume that $K_A$ is chosen to be larger than $2K'_A$, the constant that apeared in the estimates on $J \Pi$ at the beginning of this section.

\medskip

The immersions in Lemma \ref{lem:unstable_addcnst_region} will represent intermediate points of the path joining $-1$ to a stable critical point $v_{\eps}$ that we will produce in Sections \ref{2M-2B}, \ref{peak}, \ref{reach_1}, \ref{final_argument}. The next one, instead, will be used as a barrier for the gradient flow argument in Section \ref{reach_1}.

Let $\eta$ be the first eigenfunction (e.g.~normalized to have unit $L^2$ norm) of the Jacobi operator $\mathcal{L}$ (on $\tilde{M}$, which we assumed is unstable) and let $\lambda$ be the associated eigenvalue:
$$\mathcal{L}(\eta):=\Delta \eta + (|A|^2+\Rc{N}(\nu,\nu))\eta = \lambda \eta, \,\,\,\, \lambda>0,$$
where $\Delta$ is the Laplace--Beltrami operator on $\tilde{M}$. By standard theory $\eta$ is smooth and never vanishing, so we will assume that it is strictly positive on $\tilde{M}$. The positiveness of $\lambda$ follows from the fact that $\iota:\tilde{M}\to N$ is an unstable minimal immersion.

Consider the one-parameter deformation $\iota_{t}:=\iota+t\varphi \nu:\tilde{M} \to N$ for $t\in(-\delta_0, \delta_0)$ for some positive $\delta_0 < \frac{\om}{\max \varphi}$ and for some $\varphi\in C^2(\tilde{M})$. Let $H_t(p)$ denote the scalar mean curvature of $\iota_{t}$ at $p\in \tilde{M}$ with respect to the choice of unit normal $\nu_t$ (along $\iota_t$) that has positive scalar product with $\nu$. Recall that the first variation of area at time $t$ is given by $-\int_{\tilde{M}}H_t \varphi d\text{vol}_{g_t}$ (where $g_t$ is the metric induced by the pull-back via $\iota_t$) and the second variation of area at $t=0$ is given by $-\int_{\tilde{M}} \mathcal{L}(\varphi)\varphi d\text{vol}_{g_0}$. Then, using $H_0\equiv 0$, we get $\left.\frac{d}{d t}\right|_{t=0} H_t(p) = \mathcal{L}(\varphi).$
Choosing $\eta$ in place of $\varphi$ we therefore get 
$$\left.\frac{d}{d t}\right|_{t=0} H_t(p) = \mathcal{L}(\eta)=\lambda\eta(p)>0.$$
(In other words, the mean curvature of $\iota_t$ will point away from $M$ when we perturb by the first eigenfunction.) From this, we see that there exists a sufficiently small $z_1>0$ such that the scalar mean curvature $H_{t,\eta}$ of $\iota+t\eta \nu$ satisfies (recall that $\min_{\tilde{M}}\eta>0$) 
\be
\label{eq:H_eta_lower_bound}
H_{t,\eta}\geq \frac{t}{2}\, \lambda \,\min_{\tilde{M}}\eta \text{ for all }t\in[0, z_1].
\ee
We choose $z_0 \in [0,z_1]$ such that $z_0 \, \max_{\tilde{M}} \eta <c_0$, where $c_0$ was chosen in Lemma \ref{lem:unstable_addcnst_region}. The embedded hypersurface in $N$ given by $\text{graph}(z_0\eta)\subset \quo$, in other words the image of the immersion $\iota+z_0 \eta \nu$, will be used as a barrier in Section \ref{reach_1}, thanks to the mean convexity property (\ref{eq:H_eta_lower_bound}).

\section{Proof of Proposition \ref{Prop:main}}
\label{proof_main}

The upshot of the forthcoming sections is the following: given $M\subset N$ as in Proposition \ref{Prop:main}, prove that there exists $\eps_0>0$ such that for any $\eps<\eps_0$ we can find continuous path (in $W^{1,2}(N)$) that joins the constant $-1$ to a strictly stable critical point of $\Ece$ and such that the maximum of $\Ece$ along this path is at most $2\mathcal{H}^n(M)-\delta$, for some $\delta>0$ that only depends on $M$ and $N$ (hence independent of $\eps$).

\subsection{Choice of $\eps$}
\label{parameters}
Let $B$ be as in Remark \ref{oss:choiceB}, $c_0$, $t_0$, $\tau$ be as in Lemma \ref{lem:unstable_addcnst_region} and $K_A$ as in the remarks that follow Lemma \ref{lem:unstable_addcnst_region}. The geometric quantities $\mathcal{H}^n(B)$ and $\tau$ are relevant in the forthcoming construction.

In the following sections we are going to exhibit, for every sufficiently small $\eps$, an admissible continuous path in $W^{1,2}(N)$ with $\Ece$ suitably bounded along the whole path. We will specify now an initial choice $\eps<\eps_1$, which permits the construction of the $W^{1,2}$-functions that describe the path. When we will estimate $\Ece$ along the path, we will do so in terms of geometric quantities (hence independent of $\eps$) plus errors that will depend on $\eps$. For sufficiently small $\eps$, i.e.~$\eps<\eps_2$ for a choice of $\eps_2\leq \eps_1$ to be specified, these errors will be $\leq C (\eps|\log\eps|)$, where $C>0$ is independent of $\eps$; we will not keep track of the constants and will instead write $O(\eps|\log\eps|)$. At the very end (see Section \ref{final_argument}), in order to make these errors much smaller than the geometric quantities, and thus have an effective estimate for $\Ece$, we will revisit the smallness choice of $\eps$: for some $\eps_3\leq \eps_2$ we will get that whenever $\eps<\eps_3$  the errors can be absorbed in the geometric quantities. Therefore for $\eps<\eps_3$, we will obtain an upper bound for $\Ece$ along the path that is independent of $\eps$.

Now we choose $\eps_1$; the choices of $\eps_2$, $\eps_3$ will be made as we proceed into the forthcoming sections. We restrict to values sufficiently small, namely $\eps_1<1/e$, so to have that $\eps|\log\eps|$ is decreasing as $\eps$ decreases; this guarantees that the conditions specified below on $\eps_1$ hold also for each $\eps<\eps_1$. The condition $\eps_1<1/e$, also gives that the $O(\eps^2)$-control that we have in Section \ref{truncations} on the approximated one-dimensional solutions are valid for all $\eps<\eps_1$. Moreover, (recall that $\eta$, $c_0$ and $z_0$ are chosen at the end of Section \ref{immersions}) we require:

\noindent \textit{(i)} $12\eps_1|\log\eps_1|< \frac{c_0}{20}$ (and implicitly $<\frac{1}{2}\om$);

\noindent \textit{(ii)} $12\eps_1|\log\eps_1|< z_0 \min_{\tilde{M}} \eta$;

\noindent \textit{(iii)} $24\eps_1|\log\eps_1|< c_0-z_0 \max_{\tilde{M}} \eta$.

\noindent Moreover, we will need to ensure a positiveness condition on the mean curvature of the level sets of the distance function to the image of $\iota+ z_0 \eta \nu$ (equivalently, to $\text{graph}(z_0\eta)$, passing to the quotient $\quo$). The signed distance $\text{dist}(\cdot, \text{graph}(z_0\eta))$ is well-defined on $\tilde{M} \times (0,\om)$ and negative at $(p,s)$ with $s<z_0\eta(p)$ and positive at $(p,s)$ with $s>z_0\eta(p)$; its modulus is the Riemannian distance to $\text{graph}(z_0\eta)$.
To begin with, consider a tubular neighbourhood of $\text{graph}(z_0\eta)$ that does not intersect $M$, and denote by $\om_1$ its semi-width. Choose $\eps_1$ small enough to have $12\eps_1|\log\eps_1|<\om_1$. Now, for $d\in [-12\eps_1|\log\eps_1|, 12\eps_1|\log\eps_1|]$, consider the smooth embedded hypersurface given by the level set $\{\text{dist}(\cdot, \text{graph}(z_0\eta))=d\}$.
Let $H_{d,z_0\eta}$ denote the scalar mean curvature of $\{\text{dist}(\cdot, \text{graph}(z_0\eta))=d\}$, with respect to the normal on $\iota+z_0\eta \nu$ that has positive scalar product with $\nu$. Recall that for $d=0$ (i.e.~on $\text{graph}(z_0\eta)$) we have that the scalar mean curvature is $\geq \frac{z_0}{2}\lambda \min_{\tilde{M}}\eta$. By continuity we therefore ensure that, for $\eps_1$ sufficiently small, we have, for all $d\in [-12\eps_1|\log\eps_1|, 12\eps_1|\log\eps_1|]$:

\noindent \textit{(iv)} $H_{d,z_0\eta}\geq  \frac{z_0}{4}\,\lambda\, \min_{\tilde{M}}\eta$. (Implicitly, we also assumed $12\eps_1|\log\eps_1|<\om_1$.)

\subsection{From $\Ece(-1)=0$ to $2 (|M|-|B|)$}
\label{2M-2B}
We will work at fixed $\eps$, with $\eps<\eps_1$. We will often use the shorthand notation $\Lambda=3|\log\eps|$.
Recall that the (geodesic) ball $B$ has radius $R$ (Remark \ref{oss:choiceB}) and $D$ is the concentric ball with radius $2R$. Let $\tilde{D}=\iota^{-1}(D)$, $\tilde{B}=\iota^{-1}(B)$.

\noindent \textit{Definition of $\chi$}. Let $\chi_0:M\to [0,1]$ be a smooth function on $M$ with compact supports contained in $D$ and such that $\chi_0=1$ on $B$ and $|\nabla \chi_0|\leq 2/R$. Let $\chi:\tilde{M}\to [0,1]$ be defined by $\chi=\iota\circ \chi_0$. Then $\chi$ is smooth and compactly supported in $\tilde{D}$, with $\chi=1$ on $\tilde{B}$ and $|\nabla \chi|\leq 2/R$.

\medskip

\noindent \textit{Definition of $f$}. 
We define $f(y,z)$ on $\tilde{M} \times [0,\om)$ as follows ($\Psi_t$ was defined in (\ref{eq:family2})):

$$f(y, z) = \Psi_{4\eps\Lambda\chi(y)}(z).$$
This function is even in $y$, therefore it descends to a function on $\quo$ and we will now check that it is there Lipschitz. Since $f$ is smooth on $\tilde{M} \times (0,\om)$ and equal to $-1$ on $\tilde{M} \times [\om/2,\om)$, we only need to check at $x\in M$. Let $B_\rho(x)\subset M$ be a gedesic ball, then we have a well-defined tubular neighbourhood of $B_\rho(x)$ of semi-width $\om$ that is diffeomorphic to $B_\rho(x) \times (-\om,\om)$ and isometric when we endow $B_\rho(x) \times (-\om,\om)$ with the Riemannian metric from $N$. The Jacobian factor measuring the distortion of this metric from the product metric is bounded by a constant that only depends on the geometric data $M, N$, therefore it suffices to prove the Lipschitz property with respect to the product metric. Using Fermi coordinates $(a,s)$ in $B_\rho(x) \times (-\om,\om)$ the expression of $f$ becomes $\Psi_{4\eps\Lambda\chi_0(a)}(s)$ because $\Psi_t$ is even for all $t$. This expression shows that $f$ is Lipschitz continuous in $B_\rho(x) \times (-\om,\om)$, with a fixed Lipschitz constant (depending on $\Psi$ and $\chi$ and on the geometric data $M, N$).

Passing $f$ to the quotient $\quo$, we can extend it to $N$ by setting it equal to $-1$ on $N\setminus \left(\quo\right)$, (since $f=-1$ on $\tilde{M} \times [\om/2,\om)$). We will denote the function defined on $N$ also by $f$, by abuse of notation. This function is Lipschitz on $N$ (and actually smooth away from $(D\setminus B)$).

\medskip

\noindent \textit{Allen--Cahn energy of $f$}. 
We will show that if we choose $\eps$ sufficiently small, then $\Ece(f)$ is controlled by twice the area of $M\setminus B$, up to errors of type $O(\eps|\log\eps|)$. Denote by $|J\Pi|$ the Jacobian determinant of the nearest point projection $\Pi$ and recall (Section \ref{immersions}) that $\left|\,\,\frac{1}{|J\Pi|(x)}-1\,\,\right|\leq   K_A\, d(x,M)$ whenever $d(x,M)\leq 4\eps\Lambda$  and $\eps<\eps_1$. 

The Allen--Cahn energy of the Lipschitz function $f$ on $N$ is $0$ outside $\quo$, since $f=-1$ there. We perform the computation in $\tilde{M} \times (0,\om)$ (removing $M$ from the domain does not affect the computation). Denote by $\nabla_y f$ the gradient of $f$ with respect to the variables $y\in \tilde{M}$. Then by definition of $f$ we have, in $\tilde{M}\times (0,\om)$:
$$\nabla_y f = \left. \frac{d}{ds}(\Psi_s)(z)\right|_{s=4\eps\Lambda\chi(y)} 4\eps\Lambda\nabla_y \chi$$ 
and since $|\nabla \chi|\leq \frac{2}{R}$, $|\frac{d}{ds}(\Psi_s)(z)| = |\Psi'(|z|+s)| \leq \frac{3}{\eps}$, this implies ($\Lambda=3|\log\eps|$)

\begin{equation}
 \label{eq:tangential_Dirichel_estimate}
\eps|\nabla_y (f)|^2 \leq \eps \frac{C}{\eps^2} \frac{\eps^2 |\log\eps|^2}{R^2}= \frac{C \eps |\log\eps|^2}{R^2}. 
\end{equation}
(Here $C=(48\cdot 3)^2\cdot C'$, where $C'$ depends on $(\mathscr{g}_0)^{-1}\mathscr{g}_s$.) By the coarea formula (the metric is the one induced by the pull-back from $N$) the Allen--Cahn energy of $f$ is then given by
\be
\label{eq:comp_energy}
\int_{\tilde{B}} \left( \int_{(0, {\om})}\frac{1}{|J\Pi|}\left( \eps \left|\frac{\p}{\p z} f\right|^2 + \frac{W(f)}{\eps} \right)  dz\right) dy + 
\ee
$$+\int_{\tilde{D}\setminus \tilde{B}} \left( \int_{(0, {\om})}\frac{1}{|J\Pi|}\left( \eps \left|\frac{\p}{\p z} f\right|^2 + \frac{W(f)}{\eps} \right)  dz\right) dy +$$ $$+ \int_{(\tilde{D}\setminus \tilde{B})\times (0,{\om})}   \eps |\nabla_y f|^2 $$
$$+ \int_{\tilde{M}\setminus \tilde{D}}\left(\int_{(0, {\om})}  \eps \left|\frac{\p}{\p z} f\right|^2 + \frac{W(f)}{\eps} dz\right)dy .$$
The first term vanishes because $f=-1$ on the domain of integration. Thanks to (\ref{eq:tangential_Dirichel_estimate}) the third term can be made arbitrarily small by choosing $\eps$ sufficiently small; this term is $O(\eps^2|\log\eps|^3)$, since the inner integrand in non-zero only on $[0,4\eps\Lambda]$. For the second term, note that the inner integral only gives a contribution in $[0, 4\eps\Lambda]$ (as $f=-1$ on $[4\eps\Lambda,\om]$). Recalling the bounds on the Jacobian factor $|J\Pi|$ and the energy estimates on the one-dimensional profiles, see Section \ref{truncations}, we find

$$\text{second term of (\ref{eq:comp_energy})} \leq {\Hc}^n(\tilde{D}\setminus \tilde{B})\,(1+4\eps\Lambda K_A)\frac{\Ece( \Psi_{4\eps\Lambda\chi(y)})}{2}\leq $$ $$\leq {\mathcal{H}}^n(\tilde{D}\setminus \tilde{B})\,(1+4\eps\Lambda K_A)(1+O({\eps}^2)).$$ 
Arguing similarly for the fourth term, we get the upper bound
$$\text{fourth term of (\ref{eq:comp_energy})} \leq (1+4\eps\Lambda K_A)\,\left({\Hc}^n(\tilde{M})-{\Hc}^n(\tilde{D})\right)(1+O({\eps}^2)).$$
Recall that $f$ depends on $\eps$, although for notational convenience we do not explicit the dependence; we can produce $f=f^{\eps}$ as above for every $\eps<\eps_1$. The estimates obtained contain leading terms, independent of $\eps$, and errors depending on $\eps$. For a sufficiently small choice of $\eps_2\leq \eps_1$, all the errors above, for $\eps<\eps_2$, are of the type $O(\eps|\log\eps|)$. Putting together the previous estimates we conclude that, for $\eps<\eps_2$,
$$\Ece(f) \leq 2\left({\Hc}^n(M)-{\Hc}^n(B)\right)+O(\eps|\log\eps|).$$  

\medskip

\noindent \textit{Path to $-1$}. We will now exhibit a continuous (in $r$) path $\{f_r\}_{r\in[0,4\eps\Lambda]}$, with $f_r\in W^{1,2}(N)$ for all $r$ (actually, $f_r\in W^{1,\infty}(N)$), that starts at $f_0=f$ and ends at $f_{4\eps\Lambda}\equiv -1$.
Recall that $f=-1$ outside $\quo$, so we set $f_r=-1$ in $N\setminus (\quo)$ for every $r$. In order to define $f_r$ in $\quo$, we will give a definition in $\tilde{M} \times [0,\om)$, taking care that it is even on $\tilde{M}$ and therefore passes to the quotient. (Again, by abuse of notation we call $f_r$ both the function on $\quo$ and the one on the double cover.) For $(y,z)\in \tilde{M} \times [0,\om)$ we define $f_r$, for $r\in[0,4\eps\Lambda]$ by

\be
\label{eq:pathto-1}
f_r(q,z)=\Psi_{4\eps\Lambda\chi(q)+r}(z).
\ee
For $r=4\eps\Lambda$ this function becomes constantly $-1$. We can check the Lipschitz property of $f_r$ on $N$ as done for $f$ earlier, by noticing that on $B_\rho(x) \times (-\om,\om)$ (for $B_\rho(x) \subset M$) we have the coordinate expression $f_r (a,z)=\Psi_{4\eps\Lambda\chi(a)+r}(z)$ thanks to the fact that $\Psi_t$ is even. This expression shows that for every $r\in[0,4\eps\Lambda]$ the function $f_r$ is smooth on $N\setminus M$ and globally Lipschitz.

To visualize the evolution in $r$, recall from (\ref{eq:family2}) that, for every $y$, the two half profiles $\left.\Psi_{4\eps\Lambda\chi(y)}(z)\right|_{\{z>0\}}$ and $\left.\Psi_{4\eps\Lambda\chi(y)}(z)\right|_{\{z<0\}}$ in the normal direction to $T_y M$ move towards each other at unit speed (creating a non smooth point at $z=0$ on $M$, where the regularity of $f_r$ is just Lipschitz).

The same computation performed in (\ref{eq:comp_energy}), this time on $f_r$, shows that 
\be
\label{eq:energy_path_1}
\Ece(f_r)\leq (1+O(\eps|\log\eps|))2 ({\Hc}^n(M)-{\Hc}^n(B))
\ee
for every $r$, i.e.~the energy stays below $2({\Hc}^n(M)-{\Hc}^n(B))+O(\eps|\log\eps|)$ along this path. (Moreover it reaches $0$ at the end of the path.) This follows immediately upon noticing that, using the coarea formula as in (\ref{eq:comp_energy}), this time for $f_r$, the inner integrands that we find are controlled by those computed for $f$, since for every $t_1$ we have $\int_0^{\om} \Psi_{t_1+r} \leq \int_0^{\om} \Psi_{t_1}$. The errors are of the form $O(\eps|\log\eps|)$ for all $\eps<\eps_2$ for some suitably small choice of $\eps_2\leq \eps_1$. 

\begin{oss}
The choice of $\eps_2$ will be made several times in the forthcoming sections, always with the scope of making the errors controlled by $C\eps|\log\eps|$ with $C$ independent of $\eps \in (0,\eps_2)$. I should be kept in mind that the specific value $\eps_2$ might change from one instance to the next, however we make finitely many choices, therefore we implicitly assume that the correct $\eps_2$ is the smallest of all. From now on, this remark will be tacitly applied every time we say that the errors are of the form $O(\eps|\log\eps|)$ for all $\eps<\eps_2$, for some suitably small choice of $\eps_2$.
\end{oss}

\subsection{Lowering the peak}
\label{peak}

In this section we construct the portion of our path that ensures the upper bound on the maximum along the path. We keep writing $\Lambda=3|\log\eps|$. To begin with, recall the definition of signed distance (to $M_\varphi$) given in Section \ref{immersions} and note that in the case $\varphi=2\eps\Lambda$ we have that $\text{dist}(x,M_{2\eps\Lambda})$, which was defined on $\tilde{M}  \times (0,\om)$, extends by continuity to $\tilde{M}  \times \{0\}$ with value $-2\eps\Lambda$.

The definition of $f$ in Section \ref{2M-2B} can therefore equivalently be given as follows at $x=(y,z)\in \tilde{M}  \times [0,\om)$:
$$f(x) = \OHet^{\eps}_{4\eps\Lambda\chi(y)}(-\text{dist}(x,M_{2\eps\Lambda})),$$
where $\OHet^{\eps}_{s}(\cdot)=\OHet^{\eps}(\cdot -s)$ and $\chi$ 
is as in the previous section. Note that the expression is even in $y$, therefore $f$ passes to the quotient $\quo$. (The signed distance $\text{dist}(x,M_{2\eps\Lambda})$ passes to the quotient as a smooth function on $\quo$, this provides an alternative way to chech that $f$ is Lipschitz in $\quo$). Recall that $f=-1$ in $N\setminus \quo$.

\medskip

Recall the choice of $\tilde{\phi}$ in Lemma \ref{lem:unstableregion} and Remark \ref{oss:def_immersions}, with $\text{supp}\tilde{\phi} \cap \tilde{D} = \emptyset$. We will now produce a continuous path $g_t$, $t\in [0,t_0]\to g_t \in W^{1,2}(N)$, with the property that $g_0=f$ and $g_t$ is an Allen--Cahn approximation of $\iota_B + (2\eps\Lambda+t\tilde{\phi})\nu$. Under the identification of $M\times [0,\om)$ with $\quo$, this is the hypersurface
$$\text{graph}\left(\left. (2\eps\Lambda+t\tilde{\phi} )\right|_{\tilde{M}_{B}}\right) \subset \tilde{M}\times (0,\om),$$ where $t\in[0,t_0]$. Subsequently, we will produce a continuous path $g_{t_0+s}$, $s\in[0,1]$, that starts at $g_{t_0}$ and ends at a function that is an Allen--Cahn approximation of $$\text{graph}\left( (2\eps\Lambda+t_0\tilde{\phi} ) \right).$$ The latter path starts from the endpoint of the former and has the geometric effect of ``closing the hole'' at $B$.

\begin{oss}
 \label{oss:extend_unsigned_distance}
Recall from Section \ref{immersions} the notation $M_{C+t\tilde{\phi}}$
and 
the notion of signed distance to $M_{C+t\tilde{\phi}}$.
The signed distance $\text{dist}(x,M_{2\eps\Lambda+t\tilde{\phi}})$ is defined on  $\tilde{M}  \times (0,\om)$. 
We point out the following facts. Let $x\in  \tilde{M} \times \{0\}$ and $x_j\to x$, $x_j\in \tilde{M}  \times (0,\om)$ (so that the signed distance is negative on $x_j$); then $\limsup_{j\to \infty} \text{dist}(x_j,M_{2\eps\Lambda+t\tilde{\phi}}) \leq -2\eps\Lambda$. Moreover, $\text{dist}(x,M_{2\eps\Lambda+t\tilde{\phi}})$ extends continuously, with value $-2\eps\Lambda$, to $\left(\tilde{M}\setminus \text{supp}(\tilde{\phi})\right)\times \{0\}$. In particular, this continuous extension is valid on $\tilde{D}\times \{0\}$.
\end{oss}

\medskip

\noindent \textit{Definition of $g_t$}. As done earlier, we will define the functions on the double cover of $\quo$, taking care that they pass to the quotient and are there Lipschitz. Again, we do not distinguish the notation for the functions on $\quo$ and on the double cover. For $t\in [0,t_0]$, we define

$$g_t(x)=\OHet^{\eps}_{4\eps\Lambda\chi(y)}(-\text{dist}(x,M_{2\eps\Lambda+t\tilde{\phi}}))\,\,\,\text{ for } x\in \tilde{M}\times (0,\om).$$
This function is smooth on $\left(\quo\right) \setminus M$ and can be extended smoothly to $N\setminus M$ by setting it equal to $-1$ on $N\setminus \left(\quo\right)$.

We can now check that (for every $t$) $g_t$ extends continuously across $M$. On the support of $\chi$, which is contained in $\tilde{D}$, we have, thanks to Remark \ref{oss:extend_unsigned_distance}, that $-\text{dist}(x,M_{2\eps\Lambda+t\tilde{\phi}})$ is continuous and takes value $2\eps\Lambda$: therefore $g_t$ extends continuously, with value $\OHet^{\eps}_{4\eps\Lambda\chi(y)}(2\eps\Lambda)$, across $D$. On the complement of $\text{supp}\chi$, on the other hand, we have $\liminf_{x\to \tilde{M}\times \{0\}}-\text{dist}(x,M_{2\eps\Lambda+t\tilde{\phi}})\geq 2\eps\Lambda$: this implies that, away from the support of $\chi$, i.e.~where $\OHet^{\eps}_{4\eps\Lambda\chi(y)}=\OHet^{\eps}$, the function $g_t$ extends continuously to $(\tilde{M}\setminus \tilde{D})\times \{0\}$ with value $+1$. More precisely, we can check that $g_t$ is Lipschitz on $N$, and actually smooth in the complement of $\text{supp}(\nabla \chi) \times \{0\}$. The smoothness in $N\setminus (\overline{D}\setminus B)$ is immediate since $\OHet^{\eps}$, $\text{dist}$ and $\chi$ are smooth and thanks to the fact that $\OHet^{\eps}$ has all derivatives vanishing at $\pm (2\eps \Lambda)$. So we only need to check the Lipschitz property at an arbitrary point $x\in \overline{D}\setminus B$. Let $B_\rho(x)\subset M$ be a geodesic ball in $M$ centred at $x$; we work in a tubular neighbourhood $B_\rho(x)\times (-\om, \om)$ (using Fermi coordinates $(y,s)$). Recall that $\text{supp}(\phi)$ is disjoint from $\overline{D}$. Then $g_t(y,s)=\Psi_{2\eps\Lambda\chi(y)}(s)$, which shows that $g_t$ is Lipschitz in this neighbourhood.

\medskip

The path $t\in[0,t_0] \to g_t\in W^{1,2}(N)$ is easily checked to be continuous (in $t$), moreover $g_0=f$ by the expression of $f$ given in the beginning of this section. We compute $\Ece(g_t)$, similarly to (\ref{eq:comp_energy}) of Section \ref{2M-2B}; this time we use the coarea formula with the function $\Pi_{c,t}$ and the bounds on the Jacobian of $\Pi_{c,t}$, for $c=2\eps\Lambda$, see Section \ref{immersions}. We find, for every $t$, that $\Ece(g_t) \leq \text{area of }\text{graph}\left(\left. (2\eps\Lambda+t\tilde{\phi} )\right|_{\tilde{M}_{B}}\right) + O(\eps|\log\eps|)$ whenever $\eps<\eps_2$ for a suitably small choice of $\eps_2$ and independently of $t$. Therefore (for all $\eps<\eps_2$ and for all $t\in [0,t_0]$)
\be
\label{eq:energy_path_2}
\Ece(g_t) \leq 2\left({\Hc}^n(M)-\frac{3}{4}{\Hc}^n(B)\right) + O(\eps|\log\eps|)
\ee
thanks to Lemma \ref{lem:unstable_addcnst_region}.  

\medskip

\noindent \textit{Definition of $g_{t_0+s}$: ``closing the hole at $B$''}. 
We set, for $s\in [0,1]$, for $x\in \tilde{M} \times (0,\om]$,

\be
\label{eq:gh}
g_{t_0+s}(x) = \OHet^{\eps}_{4\eps\Lambda (1-s)\chi(y)}(-\text{dist}(x,M_{2\eps\Lambda+t_0\tilde{\phi}})).
\ee
The argument that follows the definition of $g_t$ above can be repeated to show that the functions $g_{t_0+s}$ (passed to the quotient) extend to smooth functions on $N\setminus M$ (with $g_{t_0+s}=-1$ in $N\setminus \left(\quo\right)$) and extend in a Lipschitz fashion across $M$. Note that as $s$ increases from $0$ to $1$, the ``weight carried by $B$'' increases from $0$ to $2$.

The path $s\in [0,1] \to g_{t_0+s}\in W^{1,2}(N)$ is continuous in $s$ by definition. Note that $g_{t_0+1}$ is an Allen--Cahn approximation of the boundaryless hypersurface given by $\text{graph}\left( (2\eps\Lambda+t_0\tilde{\phi} )\right)$.

Moreover, with computations as those in Section \ref{2M-2B}, we can compute for every $s$ the energy $\Ece(g_{t_0+s})$. For sufficiently small $\eps_2$, we obtain that, for every $s\in [0,1]$ and for every $\eps<\eps_2$, the following holds: $\Ece(g_{t_0+s}) \leq \text{ area of graph of }\left( (2\eps\Lambda+t_0\tilde{\phi} )\right)+O(\eps|\log\eps|)$. Therefore

\be
\label{eq:energy_path_3}\Ece(g_{t_0+s}) \leq 2{\Hc}^n(M) -\tau +O(\eps|\log\eps|)
\ee
thanks to Lemma \ref{lem:unstable_addcnst_region}  (and by the initial choice $\eps_1<\frac{c_0}{20}$), uniformly in $s\in [0,1]$ and for every $\eps<\eps_2$.

\subsection{Connect to a stable critical point of $\Ece$}
\label{reach_1}

To conclude, we will produce a path (continuous in $W^{1,2}(N)$) that connects $g_{t_0+1}$ to a stable critical point of $\Ece$ (and then check that for sufficiently small $\eps$, the bumpy metric assumption implies that these stable critical points are actually strictly stable). This will be achieved by a negative gradient flow, employing a barrier $m$. To ensure the barrier condition we need to push $g_{t_0+1}$ a bit more ``away from $M$'': we define $g_{t_0+1+r}$ for $r\in [0,c_0-2\eps\Lambda]$ so that (geometrically speaking) we reach an Allen--Cahn approximation of the immersion $\iota+(c_0 + t_0\tilde{\phi})\nu$ when $r=c_0-2\eps\Lambda$. 

For $r\in [0,c_0-2\eps\Lambda]$ we define $g_{t_0+1+r}=-1$ on $N\setminus \left(\quo\right)$ and for $x\in M\times [0,\om)$

\be
\label{eq:push_more}
g_{t_0+1+r}(x) = \OHet^{\eps}(-\text{dist}(x,M_{2\eps\Lambda + r +t_0\tilde{\phi}})).
\ee
Note that this is well-defined and equal to $+1$ on $M$ thanks to Remark \ref{oss:extend_unsigned_distance} and the function is smooth on $N$ for every $r$. Moreover, the path is continuous in $r$ (in $W^{1,2}(N)$) and computing $\Ece$ again, we obtain (as done for $g_{t_0+s}$) the bound

\be
\label{eq:energy_path_push_more}
\Ece(g_{t_0+1+r}) \leq 2{\Hc}^n(M) -\tau +O(\eps|\log\eps|)
\ee
thanks to Lemma \ref{lem:unstable_addcnst_region}  (and by the initial choice $\eps_1<\frac{c_0}{20}$), uniformly in $r\in [0,c_0-2\eps\Lambda]$ and for every $\eps<\eps_2$ (for a suitable choice of $\eps_2\leq \eps_1$).

\medskip

We set $h=g_{t_0+1+c_0-2\eps\Lambda}$: by definition, $h(x)=\OHet^{\eps}(-\text{dist}(x,M_{c_0 +t_0\tilde{\phi}}))$ for $x\in {\quo}$ and $h=-1$ on $N\setminus \left(\quo\right)$. Also note that $h=+1$ on a tubular neighbourhood of $M$ of semi-width $\frac{19}{20}c_0$ , by (i) of Section \ref{parameters}.

\medskip

\textit{Barrier construction}. With this preparation, we are ready to construct the barrier $m$. This will be an Allen--Cahn approximation of $\iota+z_0\eta\nu$ (see end of Section \ref{immersions}). Recall that this is an embedded hypersurface that we also denote by $M_{z_0\eta}$ and we defined a signed distance to $M_{z_0\eta}$ (for points in $\quo\setminus M$) in Section \ref{parameters}. 
For $x$ in $\tilde{M}\times (0,\om)\equiv \left( \quo\right)  \setminus M$, we set
 
$$m(x) = \OHet^{\eps}(-\text{dist}(\cdot, \text{graph}(z_0\eta))).$$
This function extends smoothly across $M$ with value $+1$: indeed, in a tubular neighbourhood of $M$, the argument of $\Psi$ is larger than $2\eps\Lambda$ by (ii) of Section \ref{parameters} and this means that $m=+1$ in a tubular neighbourhood of $M$ (recall that $\OHet^{\eps}$ is smooth with all derivatives vanishing at $\pm 2\eps\Lambda$). Similarly we set $m=-1$ on $N\setminus \left(\quo\right)$, which is also a smooth extension since $m=-1$ in $\tilde{M}\times (c_0,\om)$ by (iii) of Section \ref{parameters}. The function $m$ is thus smooth on $N$.

\medskip

In order to use it as a barrier for the flow starting at $h$, we check that $m\leq h$. Recall that $-1\leq m,h \leq +1$ by construction.  
By (iii) of Section \ref{parameters}, on the set where $\text{dist}(\cdot, \text{graph}(z_0\eta))\leq 2\eps\Lambda$ we have $\text{dist}(\cdot,M_{c_0 +t_0\tilde{\phi}})\leq -2\eps\Lambda$. This implies that whenever $m>-1$ we have $h=+1$ and so $m\leq h$ on $N$.

\medskip

\textit{Flow starting at $m$}. We consider the negative $\Ece$-gradient flow with initial condition $m$. Since $m=-1$ on $\{|\text{dist}(\cdot, \text{graph}(z_0\eta))|\geq 2\eps\Lambda\}$ we get 
$$-\ca{E'}{\eps}(m)=0 \text{ on } \{|\text{dist}(\cdot, \text{graph}(z_0\eta))|\geq 2\eps\Lambda\}.$$
To compute the first variation of $m$ with respect to $\Ece$ on $\{-2\eps\Lambda<\text{dist}(\cdot, \text{graph}(z_0\eta))<2\eps\Lambda\}$, we write the Allen--Cahn PDE for $m$. We use the chain rule to express the Laplacian in Fermi coordinates $(a,d)$ centred on $M_{z_0\eta}$ (so $m=\OHet^{\eps}(-d)$ in these coordinates). Since $|\nabla d|=1$, using the temporary notation $\eta(x)=\OHet^{\eps}(-x)$, we get $\Delta \left(\eta(d)\right) =  (\eta)''(d) + \eta'(d) \Delta d$. Moreover, $\Delta d = -H_{d,z_0\eta}$ with our convention for the choice of normal. Therefore ($\OHet^{\eps}$ and its derivatives are evaluated at $-d$)

$$-\ca{E'}{\eps}(m) = \eps \Delta m - \frac{W'(m)}{\eps}  = \eps({\OHet^{\eps}})'' - \frac{W'({\OHet^{\eps}})}{\eps} + \eps H_{d,z_0\eta} ({\OHet^{\eps}})^\prime=$$
\be
\label{eq:almost_mean_convex}
= O({\eps}^2)+\eps H_{d,z_0\eta} ({\OHet^{\eps}})^\prime.
\ee
By (iv) of Section \ref{parameters} the hypersurfaces $\{\text{dist}(\cdot, \text{graph}(z_0\eta))=d\}$ have scalar mean curvature $H_{d,z_0\eta} \geq \frac{z_0}{4}\, \lambda \, \min_{\tilde{M}}\eta$ for $d\in[-2\eps\Lambda, 2\eps\Lambda]$. Moreover $0\leq ({\OHet^{\eps}})^\prime \leq \frac{3}{\eps}$.
As a consequence, we obtain $-\ca{E'}{\eps}(m)\geq O({\eps}^2)$. Denote by $\mu=\mu_{\eps}>0$ a constant such that $|O({\eps}^2)|<\mu$, where $O(\eps^2)$ is the function appearing in (\ref{eq:almost_mean_convex}) (we will finalize the choice of $\mu$ later). Then we consider the functional $\ca{F}{\eps,\mu}(u)=\Ece(u)-\mu \int_N u$. For the first variation of $m$ with respect to $\ca{F}{\eps,\mu}$ we have 

$$-\ca{F'}{\eps,\mu}(m)> 0.$$

This allows to prove that the negative $\ca{F}{\eps,\mu}$-gradient flow $\{m_t\}$ ($t\in[0,\infty)$) with initial condition $m_0=m$ is ``mean convex'', i.e.~the condition $-\ca{F'}{\eps,\mu}(m_t)> 0$ holds for all $t\geq 0$. (By standard semi-linear parabolic theory, the flow exists smoothly for all times.) 
To see that mean convexity is preserved, one argues as follows. For notational convenience, we write for this paragraph $F_t=\eps \Delta m_t - \frac{W'(m_t)}{\eps} + \mu$ (the negative gradient). Differentiating the PDE we get that the evolution of $F_t$ is given by $\p_t F_t =   \Delta F_t - \frac{W''(m_t)}{\eps^2} F_t$. So $F_t$ is a smooth solution of $ \p_t \gamma =   \Delta \gamma - \frac{W''(m_t)}{\eps^2} \gamma$, and the constant $\gamma=0$ is also a solution to the same PDE. The condition $F>0$ is therefore preserved by the maximum principle.

The mean convexity immediately gives that $m_t:N\to \R$ increases in $t$ (since $\p_t m_t = -\frac{1}{\eps}\ca{F'}{\eps,\mu}(m_t)>0$). Moreover the limit $m_\infty$ of this flow (as $t\to \infty$) must be a stable solution of $\ca{F'}{\eps,\mu}=0$. This follows by combining the mean convexity property with the maximum principle, see for example \cite[Lemma 7.3]{B1}.

\medskip

\textit{Flow starting at $h$}. By the maximum principle, since $h\geq m$, the negative $\ca{F}{\eps,\mu}$-gradient flow $\{h_t\}$ starting at $h_0=h$ stays above $\{m_t\}$ at all times. 
Recall that there exist two constant solutions of $\ca{F'}{\eps,\mu}=0$: the constant $k$ must satisfy $W'(k) = \eps \mu k$ (and therefore $W(k)\approx c_w^2 \eps^2 \mu^2$), so one constant is slighly larger than $-1$ and the other is slighly larger than $+1$ (both are trivially stable).
Since $h <k_{\eps,\mu}$ (because $h\leq 1$) and $k_{\eps,\mu}$ is stationary, we get that $h_t\leq k_{\eps,\mu}$ for all $t$. Moreover, since $m_t\geq m$ (and $m=1$ on $M$), the limit of $h_t$ cannot be the constant close to $-1$ and we thus have two options, as $t\to \infty$:

\noindent (a) $h_t$ converges to the constant $k_{\eps,\mu}$ (which solves $\ca{F'}{\eps,\mu}=0$);

\noindent (b) $h_t$ converges to a (stable) non-constant solution $h_\infty$ of $\ca{F'}{\eps,\mu}=0$.

\medskip

In either case (a) or (b), using $h_\infty$ as initial condition, we can run the negative $\Ece$-gradient flow $\{h_{\infty,\beta}\}$ for $\beta\geq 0$; the initial condition $h_{\infty,0}=h_{\infty}$ satisfies $\ca{E'}{\eps}(h_\infty)=-\mu<0$, therefore the same arguments used above proves that the sign of the first variation is preserved, i.e.~$\ca{E'}{\eps}(h_{\infty,\beta})<0$ for all $\beta\geq 0$ and the limit as $\beta\to \infty$ is a stable solution $h_{\infty,\infty}$ to $\ca{E'}{\eps}=0$. 

In case (a), the limit $h_{\infty,\infty}$ must be the constant $+1$, since $h_{\infty,0}>1$ and the only stable solutions to $\ca{E'}{\eps}=0$ are $\pm 1$. (In fact, in case (a), all times slices are constant and they converge to the constant $+1$.) By composing the two paths produced, first from $h$ to $k_{\eps,\mu_0}$ and then from $k_{\eps,\mu_0}$ to $+1$, we obtain a continuous path from $h$ to $+1$.

Let us analyse case (b). Again, we have that the flow is decreasing (by the mean convexity condition $-\ca{E'}{\eps}<0$).
By comparing with the negative $\Ece$-gradient flow starting at $m$, we will obtain that $\{h_{\infty,\infty}>0\}$ is also non-empty. To see this, recall that the mean curvature flow with initial condition $M_{z_0\eta}$ is well-defined, smooth and mean convex (we ensured the mean convexity by pushing by the first eigenfunction). Then \cite[6.5]{Ilm} (which requires a comparison argument and the construction of a suitable barrier) guarantees that, possibly after a time $\beta_0 \eps^2 |\log\eps|$ for some $\beta_0$ independent of $\eps$ (as long as $\eps$ is sufficiently small, only depending on $m$) the gradient flow starting at $m$ will stay above $1/2$ in a region of the form $\{d<-M_0\eps|\log\eps|\}$ for some $M_0>0$ independent of $\eps$, in particular it will stay $>1/2$ on a fixed tubular neighbourhood of the hypersurface $M$. Since we had ensured $h_\infty>m$, also $h_{\infty,\infty}$ must be $>1/2$ on a fixed tubular neighbourhood of $M$ (independently of $\eps$), ensuring that $\{h_{\infty,\infty}>0\}$ contains this fixed neighbourhood. 
In conclusion, in case (b) we have that $h_{\infty,\infty}$ is a stable Allen--Cahn solution with the property that there exists a non-empty open set, independent of $\eps$, contained
in $\{h_{\infty,\infty}>1/2\}$.

The function $v_{\eps}$ in Proposition \ref{Prop:main} is the constant $+1$ if case (a) arises and is $h_{\infty,\infty}$ if case (b) arises. (We showed that $h_{\infty,\infty}$ cannot be the constant $-1$.)

\medskip

\textit{Evaluation of $\Ece$ on the path}.
Let us estimate the value of $\Ece$ along this path. For this, note that $\ca{F}{\eps,\mu}$ is decreasing along the flow $\{h_t\}$, therefore $\Ece(h_t)\leq \Ece(h) + 2\mu {\Hc}^{n+1}(N)$ for all $t$: this implies that $\Ece$ is bounded above indepedently of $\eps$. More precisely, choosing $\mu=\mu_{\eps}$ to be $2\|O(\eps^2)\|_{\infty}$, where $O(\eps^2)$ is the function appearing in (\ref{eq:almost_mean_convex}), and recalling that $\Ece(h)\leq 2{\Hc}^n(M)-\tau + O(\eps|\log\eps|)$, we can absorbe $2\mu {\Hc}^{n+1}(N)$ in the $O(\eps|\log\eps|)$ for $\eps$ sufficiently small. In other words we obtain the upper bound 
\be
\label{eq:energy_path_4}
\Ece(h_t)\leq 2{\Hc}^n(M)-\tau+O(\eps|\log\eps|)
\ee
for all $t$ and for $\eps<\eps_2$. For the second part of the path ($h_{\infty,\beta}$) the energy $\Ece$ is decreasing, so the same upper bound holds.
(In case (a), one can easily check that $\Ece$ decreases from $c_w^2 \eps \mu^2 {\Hc}^{n+1}(N)$ to $0$ on the second part of the path).

\subsection{Conclusive arguments}
\label{final_argument}

In Sections \ref{2M-2B}, \ref{peak} and (in the beginning of) \ref{reach_1} we exhibited (given $M$, which also fixed $B$ and $\tau$) continuous paths in $W^{1,2}(N)$ that join $-1$ to $f$, $f$ to $g$, $g$ to $h$. Then in Section \ref{reach_1} we produced by gadient flow a path (still continuous in $W^{1,2}(N)$) that joins $h$ to a stable solution $h_{\infty,\infty}$ of $\ca{E'}{\eps}=0$ that is not the constant $-1$. We obtained the energy bounds respectively (\ref{eq:energy_path_1}), (\ref{eq:energy_path_2}), (\ref{eq:energy_path_3}), (\ref{eq:energy_path_push_more}), (\ref{eq:energy_path_4}). These are valid uniformly on the paths for all $\eps<\eps_2$.

For all sufficiently small $\eps$, composing these partial paths we obtain a continuous path in $W^{1,2}(N)$ that starts at the constant $-1$ and ends at $h_{\infty,\infty}$ and such that
$$\Ece \text{ along this path is} \leq 2{\Hc}^n(M)-\min\left\{\tau, \frac{3{\Hc}^n(B)}{4}\right\} + O(\eps|\log\eps|).$$

Choosing $\eps_3$ sufficiently small to ensure $O(\eps|\log\eps|)\leq \min\left\{\frac{\tau}{2}, \frac{\mathcal{H}^n(B)}{6}\right\}$ the above bound gives, for all $\eps<\eps_3$, that the maximum of $\Ece$ on the path is at most $2{\Hc}^n(M)-\min\left\{\frac{\tau}{2}, \frac{{\Hc}^n(B)}{2}\right\}$. This proves Proposition \ref{Prop:main}, setting $v_{\eps}=h_{\infty,\infty}$.

\begin{oss}[proof of Corollary \ref{cor:two_sided_bumpy}]
Multiplicity-$1$ convergence of critical points $u^{\eps}$ of $\Ece$ to a minimal hypersurface $M$ implies that the functions $u^{\eps}$ converge in $BV(N)$ to a function $u_\infty:N \to \{\pm1\}$, with the property that $M=\p \{u_\infty=+1\}$. In particular, the inner normal to $\{u_\infty=+1\}$ provides a global normal to $M$, so $M$ is two-sided.
\end{oss}
 
\begin{oss}
The multiplicity-$1$ convergence guarantees that the Morse index of $M$ is at most $1$ and that only way for it to vanish is that $M$ has non trivial nullity. However, with a bumpy metric, the nullity has to be trivial, hence the Morse index of $M$ is $1$. (The lower semi-continuity of the Morse index also follows from \cite{Gas} or \cite{Hie}.)
\end{oss}

\begin{oss}
 \label{oss:optimalpath}
We sketch how to prove the statement in Remark \ref{oss:index_u_eps}. Since $V^{u_{\eps}}$ converge (upon extraction of a sequence) to $|M|$ (with multiplicity $1$), and $M$ is strictly stable (and two-sided) by the bumpy metric assumption, then the nullity of $u_{\eps}$ has to be $0$ for all sufficiently small $\eps$ (the nullity is upper-semi-continuous as $\eps\to 0$). Then $u_{\eps}$ has Morse index $1$ and we can perturb it by its first eigenfunction (in both directions) and then use a negative gradient flow (for $\Ece$) to produce a continuous path that joins the constant $-1$ to a stable solution, which once again has to be strictly stable for sufficiently small $\eps$. 
\end{oss}

\appendix
\section{Passing stability to the limit for the double cover}
\label{double_cover}
The proof of Lemma \ref{lem:stable_implies_strictly_2} will be a consequence of the following observation.

\begin{lem}
\label{lem:stable_implies_strictly}
Let $N$ be a Riemannian manifold of dimension $n+1$, and let $v_{{\eps}_i}:N\to \R$ be stable solutions to $\ca{E'}{{\eps}_i}=0$, with $V^{v_{{\eps}_i}}\to V$ ($V^{v_{{\eps}_i}}$ are the varifolds associated to $v_{{\eps}_i}$) and $V=\sum \theta_\alpha |M_\alpha|$, where $\theta_\alpha \in \N$ and $M_\alpha$ is a smoothly embedded minimal hypersurface with $\text{dim}_{\mathcal{H}}(\overline{M}\setminus M)\leq n-7$. Then the oriented double cover of $M_\alpha$ is stable for each $\alpha$.
\end{lem}

Note that the regularity of $\spt{V}$ is not an assumption: it follows from the fact that any limit will be a stationary integral varifold by \cite{HT} satisfiying a stability condition and therefore its support must be smooth away from a codimension-$7$ set by \cite{TonWic}, \cite{Wic}.

\begin{proof}[proof of Lemma \ref{lem:stable_implies_strictly}]
Denoting by $V^{i}=V^{v_{{\eps}_i}}$ the associated varifolds, we have $V^i\to V$. Let $R$ be the smoothly embedded part of a connected component of $\spt{V}$. Then $R$ is a stable minimal hypersurface and more precisely, if $R$ carries multiplicity $\theta \in \N$ we have that for every $f\in C^2(R)$ the following inequality holds: $\int f^2 (|A|^2+\text{Ric}_N(\nu,\nu)) \theta d\mathcal{H}^n\res R \leq \int |\nabla f|^2   \theta d\mathcal{H}^n\res R$. This follows from the stability inequality $\ca{E''}{\eps}\geq 0$ as in \cite{Ton} (by first extending $f$ to a $C^2$ function on $N$, employing a tubular neighbhourhood of $R$). Since $\theta$ is constant on $R$ we can write the same inequality for $\theta=2$, which amounts to the fact that the oriented double cover $\tilde{R}$ of $R$ is stable for all deformations with initial speed given by a $C^2$ even function on $\tilde{R}$. Let now $\phi$ be a $C^2$ odd function on $\tilde{R}$. Then a deformation of $\tilde{R}$ as an immersion with initial speed given by $\phi$ amounts to an ambient deformation of $R$, with initial speed given by $\phi \nu$: since both $\nu$ and $f$ are odd on $\tilde{R}$, this is a well-defined $C^2$ vector field on $R$ (in particular, if $R$ is one-sided then this vector field must vanish). Then the second variation of area is non-negative along this deformation by using \cite{Gas} to pass to the limit the stability condition for $v_{{\eps}_i}$. 

Next, we consider an arbitrary $\phi \in C^2(\tilde{R})$. We consider the deformation of the immersion $\iota:\tilde{R}\to N$ (induced by the standard $2-1$ projection) given by $\text{exp}_{\iota(p)}(\iota(p)+t\phi(p)\nu(p))$, for $t\in (-\delta, \delta)$ and $\nu$ a choice of unit normal to the immersion. The second variation of area computed at $t=0$ along this deformation is given by 
\be
\label{eq:split_odd_even}
\int |\nabla \phi|^2 - \phi^2 (|A|^2+\text{Ric}_N(\nu,\nu))\,\, d\mathcal{H}^n\res \tilde{R}. 
\ee
Consider the involution $i:\tilde{R} \to \tilde{R}$ that sends each point to the only other point with the same image via $\iota$. Then by writing $\phi_e=\frac{\phi+\phi(i)}{2}$ and $\phi_o=\frac{\phi-\phi(i)}{2}$ we obtain $\phi=\phi_e+\phi_o$ with $\phi_e$ even on $\tilde{R}$ and $\phi_o$ odd on $\tilde{R}$. We expand and rewrite (\ref{eq:split_odd_even}) as follows:

\be
\label{eq:split_odd_even_2}
\int |\nabla \phi_e|^2 +|\nabla \phi_e|^2- \phi_e^2 (|A|^2+\text{Ric}_N(\nu,\nu))-\phi_o^2 (|A|^2+\text{Ric}_N(\nu,\nu))\,\, d\mathcal{H}^n\res \tilde{R}.
\ee
Here we used the fact that $\int \nabla \phi_e \nabla \phi_o \,\, d\mathcal{H}^n\res \tilde{R} =0$ because $\nabla \phi_o\nabla \phi_e$ is an odd function, and the fact that $\phi_e \phi_o$ is odd and $|A|^2+\text{Ric}_N(\nu,\nu)$ is even, therefore the mixed product $\phi_e \phi_o (|A|^2+\text{Ric}_N(\nu,\nu))$ integrates to $0$.

In conclusion, (\ref{eq:split_odd_even_2}) shows that the second variation of area for $\iota$ along the deformation induced by $\phi \nu$ is the sum of the second variation induced by $\phi_e \nu$ and the second variation of $2|R|$ induced by the ambient vector field (in a tubular neighbhourhood of $R$) given by $\phi_o \nu$. As we saw above that both of these are non-negative, we conclude that the double cover $\tilde{R}$ of $R$ is stable.
\end{proof}

\begin{proof}[proof of Lemma \ref{lem:stable_implies_strictly_2}]
Arguing by contradiction, let $v_{{\eps}_i}$ for ${\eps}_i\to 0$ be stable critical points of $\ca{E}{{\eps}_i}$ and assume that for every $i$ strict stability fails. 

Let $V^{i}=V^{v_{{\eps}_i}}$ denote the associated varifolds, we have $V^i\to V$ (subsequentially) and $\spt{V}$ is everywhere smoothly embedded. By Lemma \ref{lem:stable_implies_strictly} each connected component $R$ of $\spt{V}$ is a smoothly embedded closed minimal hypersurface with stable double cover. By the bumpy metric assumption, the double cover is strictly stable. Then \cite{GMN} applies to give that $V^{v_{{\eps}_i}}$ converge with multiplicity $1$ to their varifold limit $V$. The multiplicity-$1$ convergence implies that the nullity is upper-semi continuous in the $\eps\to 0$ limit, therefore $v_{{\eps}_i}$ must be strictly stable for sufficiently large $i$, contradiction.
\end{proof}

\end{document}